\documentclass[a4paper,twoside,10pt]{article}

\usepackage[a4paper,left=3cm,right=3cm, top=3cm, bottom=3cm]{geometry}
\usepackage[latin1]{inputenc}
\usepackage{mathrsfs}
\usepackage{graphicx}
\usepackage{epstopdf}
\usepackage{xcolor}
\usepackage{amsmath}
\usepackage{amsthm}
\usepackage{amssymb}
\usepackage{stmaryrd}
\usepackage{float}
\usepackage{bigints}
\usepackage{cite}
\usepackage{color}
\usepackage[abs]{overpic}
\usepackage[font=footnotesize,labelfont=bf]{caption}
\usepackage{cases}
\usepackage{tikz}
\usepackage{algorithm,algpseudocode}
\usepackage{rotating}
\usepackage{blkarray}
\usetikzlibrary{matrix,calc,arrows}
\usepackage[author={Lorenzo}]{pdfcomment}
\usepackage{soul,xcolor}
\usepackage{verbatim}
\usepackage{graphicx}
\usepackage{subcaption}
\usepackage{algorithm}
\usepackage{pifont}
\usepackage{multirow}

\restylefloat{table}
\theoremstyle{plain}
\newtheorem{thm}{Theorem}[section]

\newtheorem{lem}[thm]{Lemma}

\theoremstyle{definition}

\theoremstyle{remark}
\newtheorem{remark}{Remark}

\newcommand{\eremk}{\hbox{}\hfill\rule{0.8ex}{0.8ex}}

\newcommand{\p}{p}
\newcommand{\h}{h}
\newcommand{\hE}{\h_\E}
\newcommand{\hF}{\h_\F}
\newcommand{\he}{\h_\e}

\newcommand{\E}{K}
\newcommand{\F}{F}

\newcommand{\e}{e}

\newcommand{\taun}{\mathcal T_n}
\newcommand{\tautilden}{\widetilde{\mathcal T}_n}
\newcommand{\tautildenE}{\tautilden^\E}

\newcommand{\EE}{\mathcal E^\E}

\newcommand{\Vbf}{\mathbf V}

\newcommand{\Vbfn}{\Vbf_n}

\newcommand{\VbfnE}{\Vbfn(\E)}

\renewcommand{\a}{a}
\newcommand{\aE}{\a^\E}
\renewcommand{\b}{b}
\newcommand{\bE}{\b^\E}

\newcommand{\ubf}{\mathbf u}
\newcommand{\ubfn}{\ubf_n}

\newcommand{\vbf}{\mathbf v}
\newcommand{\vbfn}{\vbf_n}

\newcommand{\Nbb}{\mathbb N}
\newcommand{\Rbb}{\mathbb R}
\newcommand{\Pbb}{\mathbb P}
\newcommand{\nablabold}{\boldsymbol \nabla}
\newcommand{\Deltabold}{\boldsymbol \Delta}
\newcommand{\Piboldnabla}{\boldsymbol\Pi^\nabla_\p}
\newcommand{\PiboldnablaF}{\boldsymbol\Pi^{\nabla,\F}_\p}
\newcommand{\Piboldzpmt}{\boldsymbol \Pi^0_{\p-2}}
\newcommand{\PiboldzpF}{\boldsymbol \Pi^{0,\F}_{\p}}

\newcommand{\Piboldzperp}{\boldsymbol \Pi^0_{\perp, \p-3}}
\newcommand{\Piboldzwedge}{\boldsymbol \Pi^0_{\wedge, \p-3}}
\newcommand{\q}{q}

\newcommand{\qbf}{\mathbf \q}

\newcommand{\qbfp}{\qbf_\p}
\newcommand{\qbfpF}{\qbf_\p^\F}
\newcommand{\qbfppmtF}{\qbf_{\p,\p-2}^\F}

\newcommand{\Ccal}{\mathcal C}
\newcommand{\boldalpha}{\boldsymbol \alpha}
\newcommand{\boldbeta}{\boldsymbol \beta}
\newcommand{\boldbetatilde}{\widetilde{\boldsymbol \beta}}
\newcommand{\boldgamma}{\boldsymbol \gamma}
\newcommand{\bolddelta}{\boldsymbol \delta}
\newcommand{\bolddeltatilde}{\widetilde{\boldsymbol \delta}}
\newcommand{\mboldalphabold}{\mathbf m_{\boldalpha}}

\newcommand{\nbf}{\mathbf n}
\newcommand{\nbfE}{\nbf_\E}

\newcommand{\ubfpi}{\ubf_\pi}

\newcommand{\ubfI}{\ubf_I}

\newcommand{\dof}{\text{dof}}
\newcommand{\dofB}{\text{dof}^{\text{B}}}
\newcommand{\dofperp}{\text{dof}^{\perp}}
\newcommand{\dofdiv}{\text{dof}^{\text{div}}}
\newcommand{\fbf}{\mathbf f}

\let\div\relax
\DeclareMathOperator{\div}{div}
\DeclareMathOperator{\rot}{rot}
\DeclareMathOperator{\curlbold}{\textbf{curl}}

\newcommand{\malphabold}{m_{\boldalpha}}
\newcommand{\mbetabold}{m_{\boldbeta}}
\newcommand{\mbetaboldtilde}{m_{\boldbetatilde}}
\newcommand{\mgammabold}{m_{\boldgamma}}
\newcommand{\mdeltabold}{m_{\bolddelta}}
\newcommand{\mdeltaboldtilde}{m_{\bolddeltatilde}}

\newcommand{\SE}{S^\E}
\newcommand{\SED}{S^\E_D}

\newcommand{\qpmo}{q_{\p-1}}

\newcommand{\Ndof}{N_{dof}}

\newcommand{\xbf}{\mathbf x}
\newcommand{\xperp}{\mathbf x^{\!\perp}}

\newcommand{\qpmth}{q_{\p-3}}
\newcommand{\qbfpmtw}{\mathbf q_{\p-2}}

\newcommand{\qbfpmth}{\mathbf q_{\p-3}}

\newcommand{\qbfboldalpha}{\mathbf q_{\boldalpha}}
\newcommand{\qbfboldalphaF}{\mathbf q_{\boldalpha}^\F}
\newcommand{\Norm}[2]{\Vert #1 \Vert_{#2}}
\newcommand{\SemiNorm}[2]{\vert #1 \vert_{#2}}
\newcommand{\Phibold}{\boldsymbol \Phi}
\newcommand{\phibold}{\boldsymbol \varphi}
\newcommand{\phiboldB}{\phibold^B}
\newcommand{\phiboldperp}{\phibold^\perp}
\newcommand{\phibolddiv}{\phibold^{\text{div}}}
\newcommand{\psibold}{\boldsymbol \psi}
\newcommand{\psiboldB}{\psibold^B}
\newcommand{\psiboldperp}{\psibold^\perp}
\newcommand{\psibolddiv}{\psibold^{\text{div}}}
\newcommand{\Ttilde}{\widetilde T}
\newcommand{\Ibf}{\mathbf I}
\newcommand{\sigmabold}{\boldsymbol \sigma}
\newcommand{\Pizpmo}{\Pi^0_{\p-1}}
\newcommand{\VbfnF}{\Vbf_n(\F)}
\newcommand{\DeltaF}{\Delta_\F}
\newcommand{\nablaF}{\nabla_\F}

\newcommand{\EcalF}{\mathcal E^\F}
\newcommand{\gbf}{\mathbf g}
\newcommand{\Abf}{\mathbf A}
\newcommand{\Bbf}{\mathbf B}
\newcommand{\Ai}{A^{(i)}}
\newcommand{\Bi}{B^{(i)}}
\newcommand{\Ci}{C^{(i)}}
\newcommand{\Aalpha}{A^{(\boldalpha)}}
\newcommand{\Balpha}{B^{(\boldalpha)}}
\newcommand{\Calpha}{C^{(\boldalpha)}}
\newcommand{\Agamma}{A^{(\boldgamma)}}
\newcommand{\Bgamma}{B^{(\boldgamma)}}
\newcommand{\Cgamma}{C^{(\boldgamma)}}
\newcommand{\deltacal}{\mathcal \delta}
\newcommand{\ah}{a_\h}
\newcommand{\ahE}{\ah^\E}
\newcommand{\ebf}{\mathbf e}
\newcommand{\BE}{B_\E}

\title{Stability and interpolation properties for Stokes-like virtual element spaces}
\author{J. Meng\thanks{School of Mathematics and Statistics, Xi'an Jiaotong University,
710049, Shaanxi, P.R. China, {\tt mengjian0710@163.com}} ,
L. Beir\~ao da Veiga\thanks{Dipartimento di Matematica e Applicazioni, Universit\`a di Milano Bicocca, 20125 Milan, Italy, {\tt lourenco.beirao@unimib.it, lorenzo.mascotto@unimib.it}}
\thanks{IMATI-CNR, 27100, Pavia, Italy}\ ,
L. Mascotto\footnotemark[2]
\thanks{Fakult\"at f\"ur Mathematik, Universit\"at Wien, 1090 Vienna, Austria, {\tt lorenzo.mascotto@univie.ac.at}}
\footnotemark[3]}
\date{}

\begin{document}

\maketitle

\begin{abstract}
\noindent We prove stability bounds for Stokes-like virtual element spaces in two and three dimensions.
Such bounds are also instrumental in deriving optimal interpolation estimates.
Furthermore, we develop some numerical tests in order to investigate the behaviour
of the stability constants also from the practical side.


\end{abstract}

\section{Introduction} \label{section:introduction}

In recent years, due to their flexibility in handling complex data features and adaptive mesh refinements, 
Galerkin methods based on polytopal meshes received an increasing attention.
The virtual element method (VEM)~\cite{BeiraoDaVeiga-Brezzi-Cangiani-Manzini-Marini-Russo:2013} is one amid the most successful of such polytopal methods.

Amongst the various problems that have been tackled with the VEM,
fluid static and dynamic problems have a prominent role.
The first paper coping with a lowest order VEM for the Stokes problem is~\cite{Antonietti-BeiraoDaVeiga-Mora-Verani:2014}.
Later, its general order conforming~\cite{BeiraoDaVeiga-Lovadina-Vacca:2017} and nonconforming versions~\cite{Cangiani-Gyrya-Manzini:2016, Liu-Li-Chen:2017}
have been discussed.
Based on that, conforming~\cite{BeiraoDaVeiga-Lovadina-Vacca:2018} and nonconforming VEMs for the Navier-Stokes problem~\cite{Liu-Chen:2019} were also introduced.
All these references are concerned with divergence free methods.

In addition to such works that represent the backbone of the VEM for fluid-type problems, other similar topics have been studied as well,
an incomplete and short list being:
mixed VEMs for the pseudo-stress-velocity formulation of the Stokes problem~\cite{Caceres-Gatica:2017};
mixed VEMs for quasi-Newtonian flows~\cite{Caceres-Gatica-Sequeira:2018};
mixed VEMs for the Navier-Stokes problem~\cite{Gatica-Munar-Sequeira:2018};
other variants of the VEM for the Darcy problem~\cite{Wang-Wang-Chen-He:2019, Vacca:2018, Caceres-Gatica-Sequeira:2017, Mora-Reales-Silgado:2021};
the analysis of the Stokes complex in the VEM framework~\cite{BeiraoDaVeiga-Mora-Vacca:2019, BeiraoDaVeiga-Dassi-Vacca:2020};
a stabilized VEM for the unsteady incompressible Navier-Stokes problem~\cite{Irisarri-Hauke:2019};
implementation details~\cite{Dassi-Vacca:2020};
a pressure robust variant of the VEM for the Stokes problem~\cite{Frerichs-Merdon:2022};
the magneto-hydrodynamic problem~\cite{BeiraoDaVeiga-Dassi-Manzini-Mascotto:2022};
the $\h\p$-version of the standard VEM for the Stokes problem~\cite{Chernov-Marcati-Mascotto:2021};
stationary quasi-geostrophic equations of the ocean~\cite{Mora-Silgado:2022};
the unsteady Navier-Stokes problem~\cite{Adak-Mora-Natarajan-Silgado:2021}.

Needless to write, other polytopal methods have been used to approximate the above problems.
For instance, we recall
the local discontinuous Galerkin method~\cite{Cockburn-Kanschat-Schoetzau-Schwab:2002};
hybrid discontinuous Galerkin schemes for the Stokes flow~\cite{Cockburn-Shi:2014,Cockburn-Sayas:2014};
hybrid discontinuous Galerkin schemes for the Navier-Stokes problem~\cite{Qiu-Shi:2016};
hybrid high-order methods~\cite{Aghili-Boyaval-DiPietro:2015,DiPietro-Krell:2018, Botti-DiPietro-Droniou:2019,Burman-Delay-Ern:2021}.

The analysis of the VEM is based on showing optimal a priori error estimates,
which are proved by means of certain
\emph{stability} and \emph{polynomial consistency} properties.
Optimal convergence is then derived based on using best polynomial and interpolation estimates.

All in all, the main difference with respect to the finite element setting resides in further employing
\begin{itemize}
    \item interpolation estimates in virtual element spaces;
    \item stability properties on a discrete bilinear form.
\end{itemize}
Several interpolation estimates for Stokes-type virtual element spaces are available in the literature;
see, e.g., \cite[Proposition~$4.2$]{BeiraoDaVeiga-Lovadina-Vacca:2017} and~\cite[Theorem~$4.1$]{BeiraoDaVeiga-Lovadina-Vacca:2018}.
Such interpolation estimates are rather technical to prove,
strongly hinge upon the definition of the local virtual element spaces,
and have been proved in two dimensions only.

On the other hand, to the best of our knowledge, stability properties for Stokes-like virtual element spaces have never been explicitly proved.
We point out that several works coping with explicit stability estimates for standard Poisson-like virtual element spaces are available;
see, e.g., \cite{Brenner-Sung:2018,BeiraoDaVeiga-Lovadina-Russo:2017, Chen-Huang:2018, Cao-Chen:2018, BeiraoDaVeiga-Chernov-Mascotto-Russo:2018}.

This paper aims at closing this theoretical gap.
Notably, we contribute to the current state of the art along the three following avenues:
\begin{enumerate}
    \item We prove explicit stability properties for Stokes-like virtual element spaces in two and three dimensions.
    To this aim, we employ two different stabilizations, namely
    one given by the inner product of the degrees of freedom
    and one in integral form,
    which is independent of the chosen degrees of freedom;
    \item Based on such stability estimates, we prove novel interpolation estimates, which deliver
    the same convergence as those already available in the literature but are much easier to prove.
    \item The above achievements are proven for regular polytopal meshes.
    Thus, we also exhibit numerical results investigating the stability constants for sequences of elements with degenerating geometry
    as well as with respect to the degree of accuracy of the method.
\end{enumerate}

\noindent In the remainder of the introduction,
we pinpoint some basic notation of the paper,
discuss the model problem we aim to approximate,
introduce sequences of regular polytopal meshes,
and detail the structure of the paper.

\paragraph*{Notation.}
Throughout, we employ standard notation for Lebesgue and Sobolev spaces.
Notably, given a domain~$D \subset \Rbb^d$, $d=1,2,3$,
$L^2(D)$ denotes the space of measurable and integrable squared functions
and~$L^2_0(D)$ its subspace consisting of functions with zero average over~$D$.
Given~$r\in\Nbb$, $H^r(D)$ denotes the Sobolev space of order~$r$,
i.e., the subspace of~$L^2(D)$ consisting of functions with integrable squared weak derivatives~$D^r\cdot$ up to order~$r$ (we conventionally set~$H^0(D) = L^2(D)$).
We endow the above spaces with the standard bilinear forms~$(\cdot,\cdot)_{r,D}:=(D^r\cdot,D^r\cdot)_{0,D}$ and (semi)norms
$\Norm{\cdot}{r,D}^2 := \sum_{\ell=0}^r(\cdot,\cdot)_{\ell,D}$
and $\SemiNorm{\cdot}{r,D}^2 := (\cdot,\cdot)_{r,D}$.
We also introduce~$H^1_0(D)$ as the space of $H^1$ functions with zero trace over the boundary~$\partial D$ of~$D$.
Noninteger order Sobolev spaces can be constructed by interpolation.

Given~$\ell \in \Nbb$, $\Pbb_\ell(D)$ denotes the space of polynomials of degree at most~$\ell$
over~$D$ and~$\Pbb_\ell(D) \setminus \Rbb := \Pbb_\ell(D) \cap L^2_0(D)$.
We use the convention~$\Pbb_{-1}(D)=\{ 0 \}$.

We recall the definition of standard differential operators in two dimensions.
For~$\E \subset \Rbb^2$, we introduce the $\rot$ and~$\curlbold$ operators as follows: given~$v:\E \to \Rbb$ and~$\vbf:\E \to \Rbb^2$,
\[
\rot \vbf := -\partial_{y} \vbf_1  + \partial_{x} \vbf_2 ,
\qquad
\curlbold v   := (\partial_{y} v, - \partial_{x} v )^T.
\]
We denote the vector product between two vectors~$\ubf$ and~$\vbf$ in three dimensions by~$\ubf \wedge \vbf$.
In other words, given~$\ebf_j \in \Rbb^3$
the vector satisfying~$\ebf_j{}_{|\ell} = \delta_{j,\ell}$,
and given the splittings
\[
\ubf = \sum_{j=1}^3 u_j \ebf_j,
\qquad
\vbf = \sum_{j=1}^3 v_j \ebf_j,
\]
the vector product is the ``determinant'' of the matrix
\[
\begin{bmatrix}
\ebf_1    &\ebf_2    &\ebf_3\\
u_1       &u_2       &u_3   \\
v_1       &v_2       &v_3
\end{bmatrix}.
\]

Finally, given two positive quantities~$a$ and~$b$, we use the short-hand-notation~$a \lesssim b$
instead of \emph{there exists a positive constant~$c$ independent of the mesh
such that} $a \le c \ b$.
We further write~$a\approx b$ if we have~$a\lesssim b$ and~$b \lesssim a$.

\paragraph*{The model problem.}
Let~$\Omega\subset\Rbb^d$, $d=2,3$, be an open domain
and~$\fbf \in [L^2(\Omega)]^d$.

As a model problem, we consider the Stokes problem
\[
\begin{cases}
\text{find } (\ubf,s) \text{ such that}\\
-\Deltabold \ubf - \nabla s = \fbf  & \text{in } \Omega\\
-\div \ubf = 0                      & \text{in } \Omega\\
\ubf = \mathbf 0                    & \text{on } \partial \Omega,
\end{cases}
\]
which in weak formulation reads as follows:
\[
\begin{cases}
\text{find } (\ubf,s) \in [H^1_0(\Omega)]^d \times L^2_0(\Omega):= \Vbf \times Q \text{ such that}\\
(\nablabold \ubf, \nablabold \vbf)_{0,\Omega}
+ (\div \vbf, s)_{0,\Omega} 
= (\fbf,\vbf)_{0,\Omega} 
& \forall \vbf \in \Vbf\\
(\div \ubf, t)_{0,\Omega} = 0 & \forall t \in Q. 
\end{cases}
\]
The well posedness of this problem is standard~\cite{Boffi-Brezzi-Fortin:2013}.

\paragraph*{Regular polytopal meshes}
Throughout, we are given sequences~$\{\taun\}$ of polytopal meshes over the domain~$\Omega$.
In~$d=2$, $\taun$ consists of conforming polygons;
in~$d=3$, $\taun$ consists of conforming polyhedra.
We denote a generic element of~$\taun$ by~$\E$;
$\partial \E$ denotes the boundary of~$\E$
with outward unit vector~$\nbfE$.
For any geometric object~$D\subset\Rbb^d$, $d=1,2,3$,
we denote its barycenter, measure, and diameter
by~$\xbf_D$, $\vert D \vert$, and~$\h_D$, respectively.

Given an element~$\E$ in three dimensions, $\partial \E$ is the union of its faces~$\F$.
Given an element in two dimensions or a face~$\F$ in three dimensions,
$\partial \F$ is the union of its edges~$\e$.

We demand standard regularity assumptions on~$\taun$:
there exists~$\rho>0$ such that
\begin{itemize}
    \item for~$d=2$,
    \begin{itemize}
        \item every polygon~$\F$ is star-shaped with respect to a disk of diameter greater than or equal to~$\rho \hF$;
        \item every edge~$\e$ satisfies $\he\ge \rho \hF$;
    \end{itemize}
    \item for~$d=3$, 
    \begin{itemize}
        \item every polyhedron~$\E$ is star-shaped with respect to a disk of diameter greater than or equal to~$\rho \hF$;
        \item every face~$\F$ of~$\E$ is star-shaped with respect to a disk of diameter greater than or equal to~$\rho \hF$;
        \item for every face~$\F$ of~$\E$ and edge~$\e$ of~$\F$,
        we have~$\he \ge \rho \hF \ge \rho^2 \he$.
    \end{itemize}
\end{itemize}
Given~$\xbf = (x_1,y_1)$, we define~$\xperp:=(x_2,-x_1)$.
We denote the set of edges of a polytope~$\E$ by~$\EE$
and the set of faces of a polyhedron~$\E$ by~$\EcalF$.

\paragraph*{Outline of the paper.}
In Sections~\ref{section:2D} and~\ref{section:3D},
after recalling the definition of Stokes-type virtual element spaces~\cite{BeiraoDaVeiga-Lovadina-Vacca:2017,BeiraoDaVeiga-Dassi-Vacca:2020},
we prove stability and interpolation properties in two and three dimensions,
respectively.
We perform the analysis for two explicit stabilizations.
In Section~\ref{section:numerical-stabilization},
we numerically check the stability properties (in 2D only) on sequences of elements with degenerating geometry and degree of accuracy.
We draw some conclusions in Section~\ref{section:conclusions}.

\section{The two dimensional case} \label{section:2D}

This section is devoted at proving stability and interpolation properties in two dimensions.
In Section~\ref{subsection:space-2D}, we recall the definition of the Stokes-like virtual element space~\cite{BeiraoDaVeiga-Lovadina-Vacca:2017}.
Stability properties are derived in Sections~\ref{subsection:stability-2D} and~\ref{subsection:stability-2D:dofi-dofi}
for a projection based and a degrees of freedom based stabilizations, respectively.
In Section~\ref{subsection:interpolation-2D},
we provide a novel and shorter proof of interpolation estimates 
based on the previously proven stability properties.

In what follows, we fix $\p \in \Nbb$, $\p \ge 2$,
which will denote the degree of accuracy of the space.
We do not consider the case~$\p=1$ as the corresponding method
is known to be unstable~\cite{Antonietti-BeiraoDaVeiga-Mora-Verani:2014}.

\subsection{Virtual element spaces in two dimensions} \label{subsection:space-2D}
Given a polygon~$\E\in\taun$,
we introduce the space
\[
\VbfnE := \{ \vbfn \in [H^1(\E)]^2 \mid \vbfn \text{ satisfies } \eqref{local-auxiliary-problem} \},
\]
where, for some~$s\in L^2_0(\E)$,
\begin{equation} \label{local-auxiliary-problem}
\begin{cases}
-\Deltabold \vbfn - \nabla s = \xperp \qpmth   & \qpmth \in \Pbb_{\p-3}(\E) \\[4pt]
\div \vbfn = \qpmo                             & \qpmo \in \Pbb_{\p-1}(\E)\\[4pt]
\vbfn{}_{|\partial \E} \in [\Ccal^0(\partial \E)]^2, \quad \vbfn{}_{|\e} \in [\Pbb_\p(\e)]^2     & \forall \e \in \EE,
\end{cases}
\end{equation}
with all equations to be intended in a weak sense.

We endow the space~$\VbfnE$ with the following set of unisolvent degrees of freedom (DoFs)~\cite{BeiraoDaVeiga-Lovadina-Vacca:2017}:
given~$\vbfn \in \VbfnE$
\begin{itemize}
\item the vector values \textbf{Dv}$_1$($\vbfn$) at the vertices of~$\E$;
\item the vector values \textbf{Dv}$_2$($\vbfn$) at the $\p-1$ internal Gau\ss-Lobatto nodes on each edge~$\e$ of~$\EE$;
\item for~$\p\ge3$, given $\{ \qbfboldalpha \}$ a basis of~$\xperp \Pbb_{\p-3}(\E)$, the ``orthogonal'' moments
\begin{equation} \label{complementary:moments}
\textbf{Dv}_3(\vbfn) := \frac{1}{\vert \E \vert} \int_\E \vbfn \cdot \qbfboldalpha;
\end{equation}
\item given $\{ \malphabold \}$ a basis of~$\Pbb_{\p-1}(\E) \setminus \Rbb$, the ``divergence'' moments
\begin{equation} \label{divergence:moments}
\textbf{Dv}_4(\vbfn) := \frac{\hE}{\vert \E \vert} \int_\E \div \vbfn \ \malphabold.
\end{equation}
\end{itemize}
As usual, we require that the bases~$\{ \qbfboldalpha \}$ and~$\{ \malphabold \}$ are invariant with respect to translations and dilations;
see~\cite{Ahmad-Alsaedi-Brezzi-Marini-Russo:2013, Dassi-Vacca:2020}.
More precisely, using the standard multi-index notation,
for given real coefficients~$\lambda_{\boldalpha}$,
such polynomials have the form
\[
\malphabold
:= \sum_{\boldalpha} \lambda_{\boldalpha} 
        \left( \frac{\xbf-\xbf_\E}{\hE}  \right)^{\boldalpha},
\qquad\qquad
\qbfboldalpha:= \xperp \malphabold.
\]
It is known~\cite{BeiraoDaVeiga-Lovadina-Vacca:2017} that, given~$\vbfn \in \VbfnE$ with known DoFs,
$\div \ \vbfn$ is explicitly computable.
Further, we can compute the two orthogonal projectors~$\Piboldzpmt: [L^2(\E)]^2 \to [\Pbb_{\p-2}(\E)]^2$
and~$\Piboldzperp: [L^2(\E)]^2 \to \xperp \Pbb_{\p-3}(\E)$ defined as follows:
for all~$\vbfn$ in~$\VbfnE$,
\begin{equation} \label{L2-projectors}
\begin{split}
& (\qbfpmtw, \vbfn - \Piboldzpmt \vbfn)_{0,\E}=0
\qquad\quad\;\;\; \forall \qbfpmtw \in [\Pbb_{\p-2}(\E)]^2, \\[4pt]
& (\xperp \qpmth, \vbfn - \Piboldzperp \vbfn)_{0,\E}=0
\qquad \forall \qpmth \in \Pbb_{\p-3}(\E).
\end{split}
\end{equation}
We can also compute the $H^1$ projector~$\Piboldnabla: [H^1(\E)]^2 \to [\Pbb_{\p}(\E)]^2$ defined as
\begin{equation} \label{H1-projectors-bulk}
(\nablabold \qbfp, \nablabold (\vbfn - \Piboldnabla \vbfn))_{0,\E}=0,
\qquad
\int_{\partial \E} (\vbfn - \Piboldnabla \vbfn) = \mathbf 0
\qquad \forall \vbfn \in \VbfnE,\;
\forall \qbfp \in [\Pbb_{\p}(\E)]^2.
\end{equation}
Functions in the virtual element space~$\VbfnE$, as well as their gradients,
are not available in closed form.
For this reason, following the virtual element gospel~\cite{BeiraoDaVeiga-Brezzi-Cangiani-Manzini-Marini-Russo:2013},
we discretize the bilinear form~$(\nabla \cdot, \nabla \cdot)_{0,\Omega}$ as follows:
for all~$\E \in \taun$,
\begin{equation} \label{local-dscrete-bf}
\begin{split}
\ahE(\ubfn,\vbfn)
& := \aE(\Piboldnabla \ubfn, \Piboldnabla \vbfn) 
  + \SE ( (\Ibf-\Piboldnabla)\ubfn, (\Ibf-\Piboldnabla)\vbfn)\\
& := (\nablabold \Piboldnabla \ubfn, \nablabold \Piboldnabla \vbfn)_{0,\E}
  + \SE ( (\Ibf-\Piboldnabla)\ubfn, (\Ibf-\Piboldnabla)\vbfn).
\end{split}
\end{equation}
The bilinear form~$\SE(\cdot,\cdot)$ is required to be coercive and continuous on~$\VbfnE\cap\ker(\Piboldnabla)$ uniformly in the mesh elements.
More precisely, on this space,
we require~$\SE(\cdot,\cdot) \approx \vert \cdot \vert^2_{1,\E}$.
Proving such an equivalence is our goal in the forthcoming sections.

The global counterpart of the space~$\VbfnE$ is constructed by a standard $H^1$-conforming DoFs-coupling.

\begin{remark} \label{remark:subtessellation}
Due to the mesh regularity assumptions,
each~$\E \in \taun$ can be split into the union of shape regular simplices
\[
\E = \cup_{T \in \tautildenE} T.
\]
This applies both in the two and the three dimensional cases.
\eremk
\end{remark}

\subsection{Stability estimates for a projection based stabilization} \label{subsection:stability-2D}
Given~$\E \in \taun$, we consider the local stabilization
\begin{equation} \label{explicit:stabilization-2D}
\SE(\ubfn,\vbfn)
:= \hE^{-2} (\Piboldzperp \ubfn, \Piboldzperp \vbfn)_{0,\E}
   + (\div \ubfn, \div \vbfn)_{0,\E}
   + \hE^{-1} (\ubfn,\vbfn)_{0,\partial \E}.
\end{equation}

We show stability estimates for the bilinear form~$\SE(\cdot,\cdot)$.
\begin{thm} \label{theorem:stability-2D}
The following stability bounds are valid:
there exist~$0 < \alpha_* < \alpha^*$ independent of~$\hE$ such that
\begin{align}
    \label{lower-boud:2D}
    & \alpha_* \SemiNorm{\vbfn}{1,\E}^2 \le \SE(\vbfn,\vbfn) \qquad \forall \vbfn \in \VbfnE,  \\
    \label{upper-boud:2D}
    & \SE(\vbfn,\vbfn) \le \alpha^* \SemiNorm{\vbfn}{1,\E}^2  \qquad \forall \vbfn \in [H^1(\E)]^2 \text{ such that } \int_{\partial \E} \vbfn= \mathbf 0.
\end{align}
Bounds~\eqref{lower-boud:2D} and~\eqref{upper-boud:2D} are valid for functions in~$\VbfnE \cap \ker(\Piboldnabla)$.
\end{thm}
\begin{proof}
We begin by proving~\eqref{lower-boud:2D}
splitting its proof into four steps.
Let~$\vbfn \in \VbfnE$ solve~\eqref{local-auxiliary-problem} and~$s \in L^2_0(\E)$ be the associated auxiliary pressure.

\medskip

\noindent\textbf{Preliminary fact~$1$.}
We first observe that an integration by parts yields
\begin{equation} \label{proof-stab-fact-1}
\Norm{\Deltabold \vbfn}{-1,\E}
:= \sup_{\mathbf 0 \ne \Phibold\in H^1_0(\E)} \frac{(\Deltabold \vbfn,\Phibold)_{0,\E}}{\Norm{\Phibold}{1,\E}}
\lesssim  \SemiNorm{\vbfn}{1,\E}.
\end{equation}

\medskip

\noindent\textbf{Preliminary fact~$2$.}
If~$\rot(\xperp \qpmth)=0$, then~$\xperp \qpmth = \mathbf 0$.
To see this, we first recall the Helmholtz-type decomposition for polynomials in two dimensions \cite[Section~2]{BeiraoDaVeiga-Brezzi-Marini-Russo:2016}
\[
\Pbb_\ell(\E)
= \nabla \Pbb_{\ell+1}(\E)
    \oplus \xperp \Pbb_{\ell-1}(\E)
    \qquad \forall \ell\in \Nbb.
\]
Using next~\cite[eq.~(2.10)]{BeiraoDaVeiga-Brezzi-Marini-Russo:2016}, we have the property
\[
\left\{ \vbf \in [\Pbb_{\ell}(\E)]^2 \right\}
\quad \Longrightarrow \quad
\left\{ \rot \vbf = 0 \Longleftrightarrow 
        \vbf = \nabla q_{\ell+1} 
        \text{ for some } q_{\ell+1} \in \Pbb_{\ell+1}(\E) \right\}.
\]
This proves that~$\xperp \qpmth = 0$.

From the regularity assumptions on the mesh,
we know that there exists a ball~$\BE$ inside~$\E$
with diameter comparable to~$\hE$.
Thus, using \cite[Lemma~$6.1$]{BeiraoDaVeiga-Lovadina-Russo:2017}
and equivalence of Sobolev norms for spaces of polynomials with finite maximum degree on a ball,
we can write
\begin{equation} \label{an-equivalence-polynomials}
\Norm{\xperp \qpmth}{0,\E} 
\lesssim \Norm{\xperp \qpmth}{0,\BE}
\lesssim \hE \Norm{\rot(\xperp \qpmth)}{0,\BE}
\le \hE \Norm{\rot(\xperp \qpmth)}{0,\E}.
\end{equation}
In light of this, we show an auxiliary bound on the~$L^2$ norm of the right-hand side in~\eqref{local-auxiliary-problem}.
Using~\eqref{an-equivalence-polynomials} and the first line in~\eqref{local-auxiliary-problem}, we write
\[
\Norm{\xperp \qpmth}{0,\E}
\lesssim \hE \Norm{\rot(\Deltabold \vbfn)}{0,\E}.
\]
Next, let~$\bE$ denote the piecewise cubic bubble function over the sub tessellation~$\tautildenE$ of~$\E$
introduced in Remark~\ref{remark:subtessellation}
such that~$\Norm{\bE}{\infty,\Ttilde}=1$ for all~$\Ttilde$ in~$\tautildenE$.
Since~$\rot(\Deltabold \vbfn)$ is a polynomial,
the following polynomial inverse estimate involving bubbles holds true:
\[
\Norm{\rot(\Deltabold \vbfn)}{0,\E}^2
\lesssim \Norm{\bE \rot(\Deltabold \vbfn)}{0,\E}^2.
\]
Integrating by parts twice,
observing that~$(\bE)^2$ and its normal derivative are zero over the boundary of each~$\Ttilde$ in~$\tautildenE$,
and using the Cauchy-Schwarz inequality and a polynomial inverse inequality twice,
we arrive at
\[
\begin{split}
\Norm{\rot(\Deltabold \vbfn)}{0,\E}^2
& \lesssim \left(\rot(\Deltabold \vbfn),(\bE)^2 \rot(\Deltabold \vbfn) \right)_{0,\E}
 = \left( \nablabold \vbfn, \nablabold \curlbold ((\bE)^2 \rot(\Deltabold \vbfn) )  \right)_{0,\E}\\[4pt]
& \lesssim \SemiNorm{\vbfn}{1,\E}\; \hE^{-2} \Norm{(\bE)^2 \rot(\Deltabold \vbfn)}{0,\E}
  \le \SemiNorm{\vbfn}{1,\E}\; \hE^{-2} \Norm{\rot(\Deltabold \vbfn)}{0,\E}.
\end{split}
\]
Combining the two above bounds yields
\begin{equation} \label{proof-stab-fact-2}
\Norm{\xperp \qpmth}{0,\E} \lesssim  \hE^{-1} \SemiNorm{\vbfn}{1,\E}.
\end{equation}

\medskip

\noindent\textbf{Preliminary fact~$3$.}
We show an upper bound on the auxiliary pressure~$s$ in~\eqref{local-auxiliary-problem}.
To this aim, we first observe that a scaled Poincar\'e inequality entails
\begin{equation} \label{H-1-L2}
\Norm{\vbf}{-1,\E} \lesssim \hE \Norm{\vbf}{0,\E} \qquad \forall \vbf \in [L^2(\E)]^2.
\end{equation}
The standard inf-sup condition~\cite{Boffi-Brezzi-Fortin:2013} for the couple~$[H^1_0(\E)]^2 \times L^2_0(\E)$
states that
\[
\Norm{s}{0,\E}
\lesssim \sup_{\mathbf 0 \ne \Phibold \in [H^1_0(\E)]^2} \frac{(s,\div \Phibold)_{0,\E}}{\SemiNorm{\Phibold}{1,\E}}.
\]
Integrating by parts,
and using the triangle inequality, the first equation in~\eqref{local-auxiliary-problem},
\eqref{H-1-L2}, \eqref{proof-stab-fact-1}, and~\eqref{proof-stab-fact-2},
we deduce
\begin{equation} \label{proof-stab-fact-3}
\begin{split}
\Norm{s}{0,\E}
& \lesssim \sup_{\mathbf 0 \ne \Phibold \in [H^1_0(\E)]^2} \frac{(\nabla s, \Phibold)_{0,\E}}{\SemiNorm{\Phibold}{1,\E}} 
  = \Norm{\nabla s}{-1,\E} 
\le \Norm{\xperp \qpmth}{-1,\E} + \Norm{\Deltabold \vbfn}{-1,\E}\\[4pt]
&  \lesssim \hE \Norm{\xperp \qpmth}{0,\E} + \SemiNorm{\vbfn}{1,\E}
  \lesssim  \SemiNorm{\vbfn}{1,\E}.
\end{split}
\end{equation}

\medskip

\noindent\textbf{Proving the lower bound~\eqref{lower-boud:2D}.}
We integrate by parts, use the first equation in~\eqref{local-auxiliary-problem}, denote the~$2\times2$ identity matrix by~$\Ibf$,
integrate by parts again, use the definition of~$\Piboldzperp$ in~\eqref{L2-projectors}, and deduce
\begin{equation} \label{stellina}
\begin{split}
\SemiNorm{\vbfn}{1,\E}^2
& = (\nablabold \vbfn, \nablabold \vbfn)_{0,\E}
= (\vbfn, (\nablabold \vbfn)\nbfE)_{0,\partial \E} + (\vbfn, \xperp \qpmth + \nabla s)_{0,\E}\\[4pt]
& = (\vbfn, (\nablabold \vbfn + \Ibf s)\nbfE)_{0,\partial \E}
    + (\Piboldzperp \vbfn,\xperp \qpmth)_{0,\E}
    - (\div \vbfn, s)_{0,\E}.
\end{split}
\end{equation}
Introduce~$\sigmabold := \nablabold \vbfn + \Ibf s$.
Due to the first equation in~\eqref{local-auxiliary-problem}, $\div \sigmabold = -\xperp \qpmth$.

Using~\eqref{proof-stab-fact-2} and~\eqref{proof-stab-fact-3} in~\eqref{stellina} yields
\[
\SemiNorm{\vbfn}{1,\E}^2
\lesssim \left( \hE^{-1} \Norm{\Piboldzperp \vbfn}{0,\E} + \Norm{\div \vbfn}{0,\E}  \right) \SemiNorm{\vbfn}{1,\E}
+ \Norm{\vbfn}{\frac12,\partial\E} \Norm{(\sigmabold) \nbfE}{-\frac12,\partial\E}.
\]
Applying the divergence trace inequality \cite[Section~$3.5.2$]{Monk:2003}
and a polynomial inverse inequality on~$\partial\E$
(recall that~$\vbfn$ is a piecewise polynomial over~$\partial \E$)
gives
\[
\begin{split}
\SemiNorm{\vbfn}{1,\E}^2
& \lesssim \left( \hE^{-1} \Norm{\Piboldzperp \vbfn}{0,\E} + \Norm{\div \vbfn}{0,\E}  \right) \SemiNorm{\vbfn}{1,\E}
         \!+\! \hE^{-\frac12}\Norm{\vbfn}{0,\partial\E} \left( \Norm{\sigmabold}{0,\E} + \hE \Norm{\div \sigmabold}{0,\E}  \right) \\[4pt]
& \lesssim \left( \hE^{-1} \Norm{\Piboldzperp \vbfn}{0,\E} \!+\! \Norm{\div \vbfn}{0,\E}  \right) \SemiNorm{\vbfn}{1,\E}
         \!+\! \hE^{-\frac12}\Norm{\vbfn}{0,\partial\E} \left( \Norm{\sigmabold}{0,\E} \!+\! \hE \Norm{\xperp \qpmth}{0,\E}  \right) \!.
\end{split}
\]
On the other hand, the triangle inequality and~\eqref{proof-stab-fact-3} entail
\[
\Norm{\sigmabold}{0,\E}
\le \SemiNorm{\vbfn}{1,\E} + \Norm{s}{0,\E}
\lesssim \SemiNorm{\vbfn}{1,\E}.
\]
Combining the two above estimates and recalling~\eqref{proof-stab-fact-2} leads to~\eqref{lower-boud:2D}.

\medskip

Next, we prove the upper bound~\eqref{upper-boud:2D}.
Let~$\vbfn \in [H^1(\E)]^2$ with zero average over~$\partial\E$.
We have to show an upper bound of the three terms on the right-hand side of~\eqref{explicit:stabilization-2D} in terms of~$\SemiNorm{\vbfn}{1,\E}$.

As for~$\hE^{-2}\Norm{\Piboldzperp \vbfn}{0,\E}^2$,
it suffices to use the stability of the $L^2$ projector and the scaled Poincar\'e inequality;
we can estimate~$\Norm{\div \vbfn}{0,\E}^2$ by~$\SemiNorm{\vbfn}{1,\E}$ directly;
we control $\hE^{-1}\Norm{\vbfn}{0,\partial \E}^2$ by using the trace and the Poincar\'e inequalities.
\end{proof}

\begin{remark} \label{remark:alternative-stab:boundary}
Let~$\{ \dofB_j \}$ denote the set of boundary DoFs in~$\VbfnE$,
i.e., the DoFs of type \textbf{Dv}$_1$ and \textbf{Dv}$_2$.
Following, e.g., \cite[eq.~$(2.14)$ with~$\alpha=0$]{Bernardi-Maday:1992}
and recalling that the edge degrees of freedom are vector values at Gau\ss-Lobatto nodes,
the boundary contribution~$\hE^{-1} (\ubfn,\vbfn)_{0,\partial \E}$ in~\eqref{explicit:stabilization-2D} can be replaced by the equivalent term
\[
\sum_{j} \dofB_j(\ubfn) \; \dofB_j(\vbfn).
\]
Thus, the stabilization in~\eqref{explicit:stabilization-2D} is spectrally equivalent to
\[
\hE^{-2} (\Piboldzperp \ubfn, \Piboldzperp \vbfn)_{0,\E}
   + (\div \ubfn, \div \vbfn)_{0,\E}
   + \sum_{j} \dofB_j(\ubfn) \dofB_j(\vbfn).
\]
\eremk
\end{remark}

\subsection{Stability estimates for the ``dofi-dofi'' stabilization} \label{subsection:stability-2D:dofi-dofi}
In this section, we prove stability estimates for the classical ``dofi-dofi'' stabilization
\begin{equation} \label{dofi-dofi:stab:2D}
\SED(\ubfn,\vbfn)
:= \sum_{j=1}^{\dim(\VbfnE)} \dof_j(\ubfn) \dof_j(\vbfn),
\end{equation}
where the set~$\{ \dof_j \}$ collects
the sets~$\{ \dofB_j \}$, $\{ \dofperp_j \}$, and~$\{ \dofdiv_j \}$
of boundary ($\textbf{Dv}_1$ and~$\textbf{Dv}_2$),
``orthogonal''  ($\textbf{Dv}_3$),
and divergence ($\textbf{Dv}_4$) DoFs, respectively.

We recall the following technical result; see, e.g., \cite[Lemma~$4.1$]{Chen-Huang:2018}.
\begin{lem} \label{lemma:equivalence-polynomials}
Let~$\{\mboldalphabold\}$ be a set of linearly independent polynomials of maximum degree~$\p\in\Nbb$ over a polygon~$\E$, that are invariant with respect to translation and dilation.
For every polynomial~$\qbfp \in [\Pbb_\p(\E)]^d$, $d=2,3$,
consider the decomposition~$\qbfp = \sum_{\boldalpha} (\overrightarrow{\qbfp})_{\boldalpha} \mboldalphabold$,
where~$\overrightarrow{\qbfp}$ denotes the coefficient vector of~$\qbfp$ with respect to the basis~$\{\mboldalphabold\}$.
Then, the following equivalence of norms is valid:
\begin{equation} \label{equivalence:polynomials:l2}
\hE \Norm{\overrightarrow{\qbfp}}{\ell^2}
\lesssim \Norm{\qbfp}{0,\E}
\lesssim \hE \Norm{\overrightarrow{\qbfp}}{\ell^2}.
\end{equation}
\end{lem}
Under the mesh regularity assumption in Section~\ref{section:introduction},
the hidden constant in~\eqref{equivalence:polynomials:l2}
are uniform with respect to the element~$\E$.

Next, we prove the following stability result, based on the techniques developed in the proof of Theorem~\ref{theorem:stability-2D}.
\begin{thm} \label{theorem:stability-2D:dofi-dofi}
The following stability bounds are valid:
there exist~$0 < \alpha_* < \alpha^*$ independent of~$\hE$ such that
\begin{align}
    \label{lower-boud:2D:dofi-dofi}
    & \alpha_* \SemiNorm{\vbfn}{1,\E}^2 \le \SED(\vbfn,\vbfn) \qquad \forall \vbfn \in \VbfnE,  \\
    \label{upper-boud:2D:dofi-dofi}
    & \SED(\vbfn,\vbfn) \le \alpha^* \SemiNorm{\vbfn}{1,\E}^2  \qquad \forall \vbfn \in \VbfnE \text{ such that } \int_{\partial \E} \vbfn= \mathbf 0.
\end{align}
Bounds~\eqref{lower-boud:2D:dofi-dofi} and~\eqref{upper-boud:2D:dofi-dofi} are valid for functions in~$\VbfnE \cap \ker(\Piboldnabla)$.
\end{thm}
\begin{proof}
We begin by proving the lower bound~\eqref{lower-boud:2D:dofi-dofi}.
Throughout, we use the same notation as in the proof of Theorem~\ref{theorem:stability-2D}.

We have
\begin{equation} \label{ABC}
\begin{split}
& \SemiNorm{\vbfn}{1,\E}^2\\
& = -\int_\E \Deltabold \vbfn \cdot \vbfn
    + \int_{\partial \E} (\nablabold \vbfn)\nbfE \cdot \vbfn
  = \int_\E (\nabla s + \xperp\qpmth) \cdot \vbfn
    + \int_{\partial \E} (\nablabold \vbfn)\nbfE \cdot \vbfn\\
& = \int_\E \xperp\qpmth \cdot \vbfn
    -\int_\E s \ \div \vbfn
    + \int_{\partial \E} (\Ibf s + \nablabold \vbfn)\nbfE \cdot  \vbfn
  = A + B + C.
\end{split}
\end{equation}
We estimate the three terms on the right-hand side separately.

Denote~$\xperp \qpmth$ by~$\gbf$ and consider the expansion
\[
\gbf := \sum_{\!\boldalpha} \overrightarrow{\gbf}_{\boldalpha} \qbfboldalpha,
\]
where~$\{ \qbfboldalpha \}$ is any basis of~$\xperp \Pbb_{\p-3}(\E)$ as in Lemma~\ref{lemma:equivalence-polynomials},
and~$\overrightarrow{\gbf}$ is the vector of the coefficients of~$\gbf$ with respect to the basis~$\{ \qbfboldalpha \}$.

We obtain
\[
\begin{split}
A 
& = \int_\E \xperp\qpmth \cdot \vbfn
  = \sum_{\boldalpha} \overrightarrow{\gbf}_{\!\boldalpha}
(\qbfboldalpha, \vbfn)_{0,\E}
   = \sum_{\boldalpha} \overrightarrow{\gbf}_{\!\boldalpha} 
  \vert\E\vert \dofperp_{\boldalpha}(\vbfn)\\
& \lesssim \hE^2 \Norm{\overrightarrow{\gbf}}{\ell^2}
               \Big(\sum_{\boldalpha} \dofperp_{\boldalpha}(\vbfn)^2 \Big)^{\frac12}
\overset{\eqref{equivalence:polynomials:l2}}{\lesssim}
        \hE \Norm{\gbf}{0,\E} \Big(\sum_{\boldalpha} \dofperp_{\boldalpha}(\vbfn)^2 \Big)^{\frac12}.
\end{split}
\]
Using the definition of~$\gbf$,
bound~\eqref{proof-stab-fact-2} gives
\[
\hE \Norm{\gbf}{0,\E}
= \hE \Norm{\xperp \qpmth}{0,\E}
\lesssim \SemiNorm{\vbfn}{1,\E}.
\]
Combining the two above bounds yields
\begin{equation} \label{bound:A}
A \lesssim 
\SemiNorm{\vbfn}{1,\E} \Big(\sum_{\boldalpha} \dofperp_{\boldalpha}(\vbfn)^2 \Big)^{\frac12}.
\end{equation}
Next, we focus on the term~$B$.
Recall that~$\div \vbfn \in \Pbb_{\p-1}(\E) \setminus \Rbb$ and set
\[
\Pizpmo s =: g = \sum_{\boldbeta} \overrightarrow{g}_{\!\boldbeta} \mbetabold,
\]
where~$\{ \mbetabold \}$ is any basis of~$\Pbb_{\p-1}(\E) \setminus \Rbb$ as in (the scalar version of) Lemma~\ref{lemma:equivalence-polynomials}.

Using~\eqref{equivalence:polynomials:l2} and~\eqref{proof-stab-fact-3},
we deduce
\begin{equation} \label{bound:B}
\begin{split}
B
& = \sum_{\boldbeta} \overrightarrow{g}_{\!\boldbeta} \int_\E \mbetabold \div \vbfn
\approx \hE \sum_{\boldbeta} \overrightarrow{g}_{\!\boldbeta} \dofdiv_{\boldbeta}(\vbfn)
    \le \hE \Norm{\overrightarrow{g}}{\ell^2} 
        \Big(\sum_{\boldbeta} \dofdiv_{\boldbeta}(\vbfn)^2 \Big)^{\frac12}\\
&  \lesssim \Norm{\Pizpmo s}{0,\E} 
\Big(\sum_{\boldbeta} \dofdiv_{\boldbeta}(\vbfn)^2 \Big)^{\frac12}
   \lesssim
   \SemiNorm{\vbfn}{1,\E} \Big(\sum_{\boldbeta} \dofdiv_{\boldbeta}(\vbfn)^2 \Big)^{\frac12}.
\end{split}
\end{equation}
The term~$C$ can be estimated as in the proof of Theorem~\ref{theorem:stability-2D}:
\begin{equation} \label{bound:C}
C \lesssim \hE^{-\frac12}\Norm{\vbfn}{0,\partial\E} \SemiNorm{\vbfn}{1,\E}.
\end{equation}
Inserting \eqref{bound:A}, \eqref{bound:B}, and~\eqref{bound:C} in~\eqref{ABC}, we deduce
\[
\SemiNorm{\vbfn}{1,\E}^2
\lesssim  \sum_{\boldalpha} \dofperp_{\boldalpha}(\vbfn)^2
          + \sum_{\boldbeta} \dofdiv_{\boldbeta}(\vbfn)^2
          + \hE^{-1}\Norm{\vbfn}{0,\partial\E}^2.
\]
Finally, the boundary contribution is spectrally equivalent to the sum of the boundary degrees of freedom squared;
see Remark~\ref{remark:alternative-stab:boundary}.
\medskip

Next, we prove the upper bound~\eqref{upper-boud:2D:dofi-dofi}.
Notably, we need to estimate three types of degrees of freedom.
Lemma~\ref{lemma:equivalence-polynomials} easily implies~$\Norm{\qbfboldalpha}{0,\E} \lesssim \hE$
and~$\Norm{\malphabold}{0,\E} \lesssim \hE$.
Since~$\vbfn$ has zero average over~$\partial \E$
and we consider scaled polynomial functions in the definition of the DoFs~\eqref{complementary:moments} and~\eqref{divergence:moments},
a scaled Poincar\'e inequality entails
a bound on the DoFs of type~$\textbf{Dv}_3$ and~$\textbf{Dv}_4$:
\[
\frac{1}{\vert\E\vert} \int_\E \vbfn \cdot \qbfboldalpha
\lesssim \hE^{-2} \Norm{\vbfn}{0,\E} \Norm{\qbfboldalpha}{0,\E}
\lesssim \SemiNorm{\vbfn}{1,\E}
\]
and
\[
\frac{\hE}{\vert\E\vert} \int_\E \div \vbfn \malphabold
\lesssim \hE^{-1} \Norm{\div \vbfn}{0,\E} \Norm{\malphabold}{0,\E}
\lesssim \SemiNorm{\vbfn}{1,\E}.
\]
As for the boundary DoFs contribution~$\textbf{Dv}_1$ and~$\textbf{Dv}_2$,
we resort to Remark~\ref{remark:alternative-stab:boundary},
write the sum of the boundary DoFs equivalently as~$\Norm{\vbfn}{0,\partial\E}$,
and use a trace inequality and a Poincar\'e inequality.
\end{proof}

\begin{remark} \label{remark:constructing-full-bf-2D}
Recalling the definition of the discrete bilinear form~$\ahE(\cdot,\cdot)$ in~\eqref{local-dscrete-bf},
stability bounds involving~$\ahE(\cdot,\cdot)$ instead of~$\SE(\cdot,\cdot)$ are analogous
to those shown in Theorems~\ref{theorem:stability-2D} and~\ref{theorem:stability-2D:dofi-dofi}
with~$\alpha_*$ and~$\alpha^*$ replaced by~$\min(1,\alpha_*)$ and~$\max(1,\alpha^*)$.
\eremk
\end{remark}

\subsection{Interpolation estimates} \label{subsection:interpolation-2D}
Interpolation estimates for Stokes-type virtual element functions are well-known;
see~\cite[Proposition~$4.2$]{BeiraoDaVeiga-Lovadina-Vacca:2017} for the standard 2D case
and~\cite[Theorem~$4.1$]{BeiraoDaVeiga-Lovadina-Vacca:2018} for the enhanced 2D case.
To the best of our knowledge, no explicit interpolation estimates are available for 3D Stokes-type virtual element functions.

In this section, we prove interpolation estimates undertaking a different avenue,
notably using the stability estimates in Theorem~\ref{theorem:stability-2D}.
This novel approach is interesting \emph{per se}.
In fact, it can be easily generalized to derive interpolation properties for other virtual element spaces,
once stability estimates are available.

For all~$\ubf \in H^{1+\varepsilon}(\E)$,
$\varepsilon>0$,
we define~$\ubfI \in \VbfnE$ as the only function satisfying
\begin{equation} \label{definition:interpolant-2D}
    \dof_j(\ubf-\ubfI) = 0
    \qquad
    \qquad \forall j=1, \dots, \dim(\VbfnE).
\end{equation}
We have the following interpolation estimates.

\begin{thm} \label{theorem:interpolation-2D}
Given~$\ubf \in [H^{s+1}(\E)]^2$, $0 < s \le \p$, and~$\ubfI$ its DoFs interpolant as in~\eqref{definition:interpolant-2D},
the following bound is valid:
\[
\Norm{\ubf-\ubfI}{0,\E} + \hE \SemiNorm{\ubf-\ubfI}{1,\E}
\lesssim \hE^{s+1} \SemiNorm{\ubf}{s+1,\E}.
\]
The hidden constant depends on the shape-regularity of the mesh and the degree of accuracy~$\p$.
\end{thm}
\begin{proof}
Let~$\ubfpi$ be~$\Piboldnabla \ubf$, being~$\Piboldnabla$ defined in~\eqref{H1-projectors-bulk}.
The triangle inequality gives
\[
\SemiNorm{\ubf-\ubfI}{1,\E}
\le \SemiNorm{\ubfpi-\ubfI}{1,\E} + \SemiNorm{\ubf-\ubfpi}{1,\E}.
\]
Thanks to standard polynomial approximation results, we only need to bound the first term on the right-hand side.
Observe that~$\ubfpi-\ubfI$ belongs to~$\VbfnE$.
Let~$\SE(\cdot,\cdot)$ be defined in~\eqref{explicit:stabilization-2D}.
Then, we use bound~\eqref{lower-boud:2D}
and Remark~\ref{remark:alternative-stab:boundary},
and write
\[
\begin{split}
\SemiNorm{\ubfpi-\ubfI}{1,\E}^2 
& \lesssim \SE(\ubfpi-\ubfI, \ubfpi-\ubfI)\\[4pt]
& \lesssim \hE^{-2} \Norm{\Piboldzperp (\ubfpi-\ubfI)}{0,\E}^2
    + \Norm{\div (\ubfpi-\ubfI)}{0,\E}^2
    + \sum_j \dofB_j(\ubfpi-\ubfI)^2.
\end{split}
\]
We prove that each of the three terms above can be bounded by an error term involving the difference~$\ubf \! - \! \ubfpi$.
To this aim, we preliminary observe that
\begin{equation} \label{furbizie}
\Piboldzperp (\ubf-\ubfI)=0,
\qquad 
\div \ubfI = \Pizpmo (\div \ubf),
\qquad
\dofB_j(\ubf-\ubfI)=0,
\end{equation}
where~$\Pizpmo: L^2(\E) \to \Pbb_{\p-1}(\E)$ is the $L^2$ scalar orthogonal projection onto~$\Pbb_{\p-1}(\E)$.

Thus, as for the bulk~$L^2$ term, we use~\eqref{furbizie}, the stability of orthogonal projections,
the fact that~$\ubf-\ubfpi$ has zero average over~$\E$, and the Poincar\'e inequality.
Next, we estimate the divergence contribution by using~\eqref{furbizie} and bounding the $L^2$ norm of the divergence by the $H^1$ seminorm.
Eventually, the boundary contribution is bounded using~\eqref{furbizie} again
and the Sobolev embedding~$[H^{1+\varepsilon}(\E)]^2 \hookrightarrow [L^{\infty}(\E)]^2$,
$\varepsilon>0$, as follows: for each boundary degree of freedom~$\dofB_j$,
\[
\dofB_j(\ubfpi-\ubfI)^2
= \dofB_j(\ubf-\ubfpi)^2
\le \Norm{\ubf-\ubfpi}{\infty,\E}^2
\lesssim \hE^{-2} \Norm{\ubf-\ubfpi}{0,\E}^2 
         + \hE^{2\varepsilon} \SemiNorm{\ubf-\ubfpi}{1+\varepsilon,\E}^2.
\]
Estimates in the~$H^1$ norm follow using standard polynomial approximation properties as in~\cite{Verfuerth:1999}.
Instead, estimates in the~$L^2$ norm are a simple consequence of Poincar\'e-type arguments
and the estimates in the $H^1$ seminorm.
\end{proof}

\begin{remark}
The present analysis assumes that the length of each edge is comparable to the diameter of the parent element;
see the second mesh condition ($d=2$) at the end of Section \ref{section:introduction}.
Nevertheless, our results could be generalized to the ``small edges'' case
by combining the present analysis with the ideas in ~\cite{BeiraoDaVeiga-Lovadina-Russo:2017, Brenner-Sung:2018}.
\eremk
\end{remark}

\section{The three dimensional case} \label{section:3D}

This section is devoted at proving stability and interpolation properties in three dimensions.
In Section~\ref{subsection:space-3D}, we recall the definition of the three dimensional Stokes-like virtual element space~\cite{BeiraoDaVeiga-Dassi-Vacca:2020}.
Stability properties are derived in Section~\ref{subsection:stability-3D}
for a projection based stabilization
(comments on the degrees of freedom based stabilization are discussed in
Remark~\ref{remark:dofi-dofi:3D}).
In Section~\ref{subsection:interpolation-3D},
we provide interpolation estimates based on the previously proven stability properties.

\subsection{Virtual element spaces in three dimensions} \label{subsection:space-3D}

\paragraph*{Virtual element spaces on faces.}
Given a polyhedron~$\E \in \taun$, on each of its faces~$\F$,
we define the~$H^1$ projector~$\PiboldnablaF: [H^1(\F)]^3 \to [\Pbb_{\p}(\F)]^3$ as
\begin{equation} \label{H1-projectors-face}
(\nablaF \qbfpF, \nablaF(\vbf - \PiboldnablaF \vbf))_{0,\F}=0,
\quad
\int_{\partial \F} \vbf - \PiboldnablaF \vbf = \mathbf 0
\quad \forall \vbf \in [H^1(\F)]^3,\;
\forall \qbfpF \in [\Pbb_{\p}(\F)]^3.
\end{equation}
Based on this, we define the nodal (enhanced) virtual element space
\begin{equation} \label{nodal-VE-space-faces-enhanced}
\VbfnF
:= \{ \vbfn \in [\mathcal C^0(\F)]^3 \mid 
        \DeltaF \vbfn \in [\Pbb_\p(\F)]^3,
        \; \vbfn{}_{|\e} \in [\Pbb_\p(\e)]^3 \; \forall \e \in \EcalF,
        \; \vbfn \text{ satisfies } \eqref{enhancing-constraint-face} \},
\end{equation}
where, given~$[\Pbb_{\p,\p-2}(\E)]^3$ the space of homogeneous vector polynomials
of degree larger than~$\p-2$ and smaller than or equal to~$\p$,
\begin{equation} \label{enhancing-constraint-face}
    \int_\F (\vbfn - \PiboldnablaF \vbfn) \qbfppmtF = 0 
    \quad 
    \forall \qbfppmtF \in [\Pbb_{\p,\p-2}(\E)]^3.
\end{equation}
We endow the space~$\VbfnF$ with the following set of unisolvent DoFs~\cite{Ahmad-Alsaedi-Brezzi-Marini-Russo:2013}:
\begin{itemize}
\item the vector values \textbf{Dv}$_1^\F$($\vbfn$) at the vertices of~$\F$;
\item the vector values \textbf{Dv}$_2^\F$($\vbfn$) at the $\p-1$ internal Gau\ss-Lobatto nodes on each edge~$\e$ of~$\EcalF$;
\item for~$\p\ge2$, given $\{ \qbfboldalphaF  \}$ a basis of~$[\Pbb_{\p-2}(\F)]^3$, the moments
\[
\textbf{Dv}^\F_3(\vbfn) := \frac{1}{\vert \F \vert} \int_\F \vbfn \cdot \qbfboldalphaF .
\]
\end{itemize}
Such DoFs allow for the computation of the projector~$\PiboldnablaF$ in~\eqref{H1-projectors-face}.
Thus, as discussed in~\cite{BeiraoDaVeiga-Brezzi-Cangiani-Manzini-Marini-Russo:2013},
the enhancing constraint~\eqref{enhancing-constraint-face}
allows for the computation of the orthogonal projector~$\PiboldzpF: [L^2(\F)]^3 \to [\Pbb_{\p}(\F)]^3$ defined as
\[
\begin{split}
(\qbfpF, \vbfn - \PiboldzpF \vbfn)_{0,\F}=0
\qquad \forall \vbfn \in \VbfnF,
\quad \forall \qbfpF \in [\Pbb_{\p}(\F)]^3.
\end{split}
\]

\paragraph*{Virtual element spaces on polyhedra.}
Given a polyhedron~$\E$, we define the space
\[
\VbfnE := \{ \vbfn \in [H^1(\E)]^3 \mid \vbfn \text{ satisfies } \eqref{local-auxiliary-problem-3D} \},
\]
where, for some~$s\in L^2_0(\E)$,
\begin{equation} \label{local-auxiliary-problem-3D}
\begin{cases}
-\Deltabold \vbfn - \nabla s = \xbf \wedge \qbfpmth & \qbfpmth \in [\Pbb_{\p-3}(\E)]^3 \\[4pt]
\div \vbfn = \qpmo                                  & \qpmo \in \Pbb_{\p-1}(\E)\\[4pt]
\vbfn{}_{|\partial \E} \in [\Ccal^0(\partial \E)]^3, \vbfn{}_{|\F} \in \VbfnF & \forall \F \in \EE,
\end{cases}
\end{equation}
with all the equations to be understood in a weak sense.

We endow the space~$\VbfnE$ with the following set of unisolvent DoFs~\cite{BeiraoDaVeiga-Dassi-Vacca:2020}:
\begin{itemize}
\item the vector values \textbf{Dv}$_1$($\vbfn$) at the vertices of~$\E$;
\item the vector values \textbf{Dv}$_2$($\vbfn$) at the $\p-1$ internal Gau\ss-Lobatto nodes on each edge~$\e$ of~$\E$;
\item for all faces~$\F \in \EcalF$, given $\{ \qbfboldalphaF  \}$ a basis of~$[\Pbb_{\p}(\F)]^3$,
the moments
\begin{equation} \label{complementary:moments-face-bis}
\textbf{Dv}^\F_3(\vbfn) := \frac{1}{\vert \F \vert} \int_\F \vbfn \cdot \qbfboldalphaF .
\end{equation}
\item for~$\p\ge3$, given $\{ \qbfboldalpha  \}$ a basis of~$\xbf \wedge [\Pbb_{\p-3}(\E)]^3$,
the bulk ``orthogonal'' moments
\begin{equation} \label{complementary:moments-3D}
\textbf{Dv}_4(\vbfn) := \frac{1}{\vert \E \vert} \int_\E \vbfn \cdot \qbfboldalpha;
\end{equation}
\item given $\{ \malphabold \}$ a basis of~$\Pbb_{\p-1}(\E) \setminus \Rbb$, the bulk ``divergence'' moments
\begin{equation} \label{divergence:moments-3D}
\textbf{Dv}_5(\vbfn) := \frac{\hE}{\vert \E \vert} \int_\E \div \vbfn \ \malphabold.
\end{equation}
\end{itemize}
As for the two dimensional case,
we require that the bases~$\{ \qbfboldalphaF \}$,
$\{ \qbfboldalpha \}$,
and~$\{ \malphabold \}$ are invariant with respect to translations and dilations;
see~\cite{Ahmad-Alsaedi-Brezzi-Marini-Russo:2013, Dassi-Vacca:2020}.

The unisolvence of the above DoFs is proved, e.g., as in~\cite{BeiraoDaVeiga-Dassi-Vacca:2020}.
Thanks to the enhancement in the definition of the virtual element spaces on faces,
we can compute the three dimensional version of the projectors in~\eqref{L2-projectors} and~\eqref{H1-projectors-bulk}.

As in the two dimensional case, the global counterpart of the space~$\VbfnE$ is constructed
by a standard $H^1$-conforming DoFs-coupling.

\subsection{Stability estimates} \label{subsection:stability-3D}
The definition of the DoFs in Section~\ref{subsection:space-3D} allows us to compute
the orthogonal projector~$\Piboldzwedge: [L^2(\E)]^3 \to \xbf \wedge [\Pbb_{\p-3}(\E)]^3$ defined as
\[
\begin{split}
(\qbfboldalpha, \vbfn - \Piboldzwedge \vbfn)_{0,\E}=0
\qquad \forall \vbfn \in \VbfnE,
\quad \forall \qbfboldalpha \in \xbf\wedge[\Pbb_{\p-3}(\E)]^3.
\end{split}
\]
We consider the local stabilization
\begin{equation} \label{explicit:stabilization-3D}
\begin{split}
\SE(\ubfn,\vbfn)
& :=   \hE^{-2} (\Piboldzwedge \ubfn, \Piboldzwedge \vbfn)_{0,\E}
        + (\div \ubfn, \div \vbfn)_{0,\E} \\[4pt]
& \quad + \sum_{\F\in\EE} \left[ \hF^{-1}  (\PiboldzpF \ubfn, \PiboldzpF \vbfn)_{0,\F}
        + (\ubfn,\vbfn)_{0,\partial \F} \right].
\end{split}
\end{equation}
We prove the following stability estimates.

\begin{thm} \label{theorem:stability-3D}
The following stability bounds are valid:
there exists~$0 < \alpha_* < \alpha^*$ independent of~$\hE$ such that, for all~$\vbfn$ in~$\VbfnE$
such that~$\int_{\partial \E} \vbfn= \mathbf 0$,
\begin{align}
\label{lower-boud:3D}
& \alpha_* \SemiNorm{\vbfn}{1,\E}^2 \le \SE(\vbfn,\vbfn) , \\
\label{upper-bound:3D}
& \SE(\vbfn,\vbfn) \le 
\alpha^* \SemiNorm{\vbfn}{1,\E}^2 .
\end{align}
\end{thm}
\begin{proof}
We begin by proving the lower bound~\eqref{lower-boud:3D}.
As in the proof of Theorem~\ref{theorem:stability-2D}, we can prove the following bounds:
for all~$\vbfn \in \VbfnE$, with ``auxiliary pressure''~$s$ and right-hand side~$\xbf\wedge\qbfpmth$ in~\eqref{local-auxiliary-problem-3D},
\[
\SemiNorm{\Deltabold \vbfn}{-1,\E} \lesssim  \SemiNorm{\vbfn}{1,\E},\qquad
\Norm{\xbf \wedge \qbfpmth}{0,\E} \lesssim  \hE^{-1} \SemiNorm{\vbfn}{1,\E},\qquad 
\Norm{s}{0,\E} \lesssim  \SemiNorm{\vbfn}{1,\E}.
\]
With this at hand, as in the 2D case, we deduce
\[
\SemiNorm{\vbfn}{1,\E} 
\lesssim \hE^{-1} \Norm{\Piboldzwedge \vbfn}{0,\E} + \Norm{\div \vbfn}{0,\E} + \Norm{\vbfn}{\frac12,\partial \E} .
\]
Thus, we only have to estimate the last term on the right-hand side.
After recalling that~$\vbfn$ has zero vector average over~$\partial\E$,
we observe the bound
\cite[eq. (2.16)]{Brenner-Sung:2018}
\[
\Norm{\vbfn}{\frac12,\partial \E}^2
\lesssim \hE \sum_{\F \in \EE} \Norm{\nablaF \vbfn}{0,\F}^2.
\]
We can estimate each face contribution by means of standard nodal virtual element inverse estimates;
see, e.g., \cite[Theorem~$2$]{BeiraoDaVeiga-Chernov-Mascotto-Russo:2018} and~\cite[Section~$3$]{Chen-Huang:2018}:
\[
\Norm{\nablaF \vbfn}{0,\F}^2
\lesssim \hF^{-2} \Norm{\PiboldzpF \vbfn}{0,\F}^2 
            + \hF^{-1} \Norm{\vbfn}{0,\partial \F}^2.
\]
The assertion follows summing over all the faces and collecting the above estimates.
\medskip

Next, we focus on the upper bound~\eqref{upper-bound:3D}.
Notably, we estimate from above the four terms on the right-hand side of~\eqref{explicit:stabilization-3D}.
Since~$\vbfn$ has zero average over~$\partial\E$, we readily have
\[
\hE^{-2} \Norm{\Piboldzwedge \vbfn}{0,\E}^2
+ \Norm{\div \vbfn}{0,\E}^2
\lesssim \SemiNorm{\vbfn}{1,\E}^2.
\]
The trace and the Poincar\'e inequalities also yield
\[
\sum_{\F\in\EE} \hF^{-1}  \Norm{\PiboldzpF \vbfn}{0,\F}^2
\lesssim \SemiNorm{\vbfn}{1,\E}^2.
\]
We are left with estimating the fourth term on the right-hand side of~\eqref{explicit:stabilization-3D}.
Using a trace inequality on each face~$\F$ yields
\[
\Norm{\vbfn}{0,\partial\F} 
\lesssim \hF^{-\frac12} \Norm{\vbfn}{0,\F}
+ \hF^{\frac12} \SemiNorm{\vbfn}{1,\F}.
\]
Inverse estimates for nodal virtual element functions,
see, e.g., \cite[Theorem~$3.6$]{Chen-Huang:2018}, entail
\[
\Norm{\vbfn}{0,\partial\F} 
\lesssim \hF^{-\frac12} \Norm{\vbfn}{0,\F}.
\]
We can apply this inverse estimate
as the restriction of~$\vbfn$ on each face~$\F$
belongs to a two dimensional nodal virtual element space.

Taking the square on both sides and summing over the faces,
then using another trace inequality,
and eventually a Poincar\'e inequality on~$\E$
give the assertion.
\end{proof}

\begin{remark} \label{remark:dofi-dofi:3D}
Also the 3D ``dofi-dofi'' stabilization
\[
\SED(\ubfn,\vbfn)
:= \sum_{j=1}^{\dim(\VbfnE)} \dof_j(\ubfn) \dof_j(\vbfn)
\]
satisfies the bounds~\eqref{lower-boud:3D} and~\eqref{upper-bound:3D}.
To see this it suffices to take the steps from Theorem~\ref{theorem:stability-3D}
and argue similarly as in the proof of Theorem~\ref{theorem:stability-2D:dofi-dofi}.
Moreover, considerations analogous to those in Remark~\ref{remark:constructing-full-bf-2D} are valid for the 3D case.
\eremk
\end{remark}

\subsection{Interpolation estimates} \label{subsection:interpolation-3D}
In this section, we prove interpolation estimates for 3D Stokes-like virtual element spaces.

For all~$\ubf \in H^{\frac32+\varepsilon}(\E)$,
$\varepsilon>0$,
we define~$\ubfI$ as the only function in~$\VbfnE$ satisfying
\begin{equation} \label{definition:interpolant-3D}
    \dof_j(\ubf-\ubfI) = 0
    \qquad
    \qquad \forall j=1, \dots, \dim(\VbfnE).
\end{equation}
We have the following interpolation estimates.

\begin{thm} \label{theorem:interpolation-3D}
Given~$\ubf \in [H^{s+1}(\E)]^3$, $1/2 < s \le \p$,
and~$\ubfI$ its DoFs interpolant as in~\eqref{definition:interpolant-3D},
the following bound is valid:
\[
\Norm{\ubf-\ubfI}{0,\E} + \hE  \SemiNorm{\ubf-\ubfI}{1,\E} 
\lesssim \hE^{s+1} \SemiNorm{\ubf}{s+1,\E}.
\]
The hidden constant depends on the shape-regularity of the mesh and the degree of accuracy~$\p$.
\end{thm}
\begin{proof}
Let~$\ubfpi$ the best vector polynomial approximation of~$\ubf$ in~$[H^1(\E)]^3$; see, e.g., \cite{Verfuerth:1999}.
As in the proof of Theorem~\ref{theorem:interpolation-2D}, we only need to estimate the energy of~$\ubfpi-\ubfI$.

In addition to the three dimensional counterpart of~\eqref{furbizie}, we also have
\begin{equation} \label{furbizie-extra}
\PiboldzpF (\ubf-\ubfI)=0 \qquad \forall \F \in \EE.
\end{equation}
Let~$\SE(\cdot,\cdot)$ be defined in~\eqref{explicit:stabilization-3D}.
Using~\eqref{lower-boud:3D}, we write
\[
\begin{split}
\SemiNorm{\ubfpi-\ubfI}{1,\E}^2
& \lesssim \SE(\ubfpi-\ubfI, \ubfpi-\ubfI)\\
& =     \hE^{-2} \Norm{\Piboldzwedge (\ubfpi-\ubfI)}{0,\E}^2
        + \Norm{\div (\ubfpi-\ubfI)}{0,\E}^2\\
&\quad  + \sum_{\F\in\EE} \left[ \hF^{-1}  \Norm{\PiboldzpF (\ubfpi-\ubfI)}{0,\F}^2 \right]
        + \sum_j \dofB_j(\ubfpi-\ubfI)^2.
\end{split}
\]
We have to bound the four terms on the right-hand side by~$\SemiNorm{\ubfpi-\ubfI}{1,\E}$.
The first, second, and fourth are dealt with by using~\eqref{furbizie}, as in the proof of Theorem~\ref{theorem:interpolation-2D}.
As for the ``new'' third term, it suffices to resort to~\eqref{furbizie-extra},
the continuity of~$\PiboldzpF$,
the trace inequality,
and standard polynomial approximation properties.

Estimates in the~$L^2$ norm are a consequence of Poincar\'e type estimates on~$\ubf-\ubfI$
and the energy estimates.
\end{proof}

\begin{remark} \label{remark:stability-enhancement}
So far, we derived stability and interpolation properties for standard Stokes-like spaces.
Following, e.g., \cite{BeiraoDaVeiga-Lovadina-Vacca:2018},
we may also consider the enhanced version of such spaces,
in the same spirit as we defined the nodal virtual elements spaces in~\eqref{nodal-VE-space-faces-enhanced}.
It is apparent that no essential modifications take place in the proof of the stability estimates.
The only difference between our setting and the enhanced one
is that in the latter we should employ polynomial inverse estimates for slightly larger polynomial degrees.
\eremk
\end{remark}

\section{Numerical validation of the stability estimates in 2D} \label{section:numerical-stabilization}
In the foregoing sections, we proved stability bounds for fixed degree of accuracy and regular polytopal meshes.
In this section, we numerically investigate the behaviour of the stability constants
from the practical side in two different scenarios:
\begin{itemize}
    \item while keeping a mesh fixed, increase the degree of accuracy~$\p$;
    \item while keeping the degree of accuracy~$\p$ fixed, consider sequences of elements with degenerating geometry.
\end{itemize}
We focus on the two dimensional case,
despite the arguments we discuss below can be generalized to three dimensions.

Since we are interested in approximating the stability constants discussed in Remark~\ref{remark:constructing-full-bf-2D},
we only need to investigate the behaviour of the minimum and maximum generalized eigenvalues of problem
\begin{equation} \label{generalized:eigenvalue-problem}
\Abf \vbf = \lambda \Bbf \vbf,
\end{equation}
where the square symmetric matrices~$\Abf$ and~$\Bbf$ are defined as follows:
for a given element~$\E\in\taun$ and the canonical basis~$\{\phibold_j\}$ of~$\VbfnE$, $j=1,\dots,\Ndof$,
$\Ndof$ being the dimension of~$\VbfnE$,
associated with the degrees of freedom
\textbf{Dv}$_1$($\cdot$),
\textbf{Dv}$_2$($\cdot$),
\textbf{Dv}$_3$($\cdot$),
and~\textbf{Dv}$_4$($\cdot$),
\[
\Abf_{i,j} = \ahE(\phibold_j, \phibold_i),
\qquad
\Bbf_{i,j} = (\nablabold \phibold_j, \nablabold \phibold_i)_{0,\E}
\qquad \forall i, j=1,\dots,\Ndof.
\]
The bilinear form~$\ahE(\cdot,\cdot)$ is computable via the degrees of freedom
following definition~\eqref{local-dscrete-bf}
and standard VEM arguments in order to compute~$\Piboldnabla$.
Therefore, we only have to compute the entries of~$\Bbf$.
This is not immediate:
the canonical basis functions are not available in closed form
since are solutions to local Stokes problems with an unknown datum for the first equation.
Thus, we need to detail how to approximate them.

We split the canonical basis~$\{ \phibold \}$ into three sets
\begin{equation} \label{splitting-canonical}
\{ \phiboldB \},\qquad\qquad
\{ \phiboldperp \},\qquad\qquad
\{ \phibolddiv \},
\end{equation}
which denote the basis functions associated with boundary degrees of freedom,
the ``orthogonal'' moments~\eqref{complementary:moments},
and the divergence moments~\eqref{divergence:moments}, respectively.

The elements of this basis are defined implicitly via the degrees of freedom,
so we cannot directly approximate them by means of any Galerkin methods
(unless resorting to some mixed formulation approach).
For this reason, we introduce a different basis~$\{\psibold\}$
of the space~$\VbfnE$ and split it into sets as those in~\eqref{splitting-canonical}:
\[
\{ \psiboldB \},\qquad\qquad
\{ \psiboldperp \},\qquad\qquad
\{ \psibolddiv \}.
\]
The elements of this preliminary basis are constructed so that they solve Stokes problems with given polynomial data,
and can therefore be approximated at any precision by means, e.g., of a finite element method on a subtriangulation of the element.

We define the three type of basis functions as follows:
given~$\{\malphabold\}$ a basis of~$\Pbb_{\p-3}(\E)$ and~$\{\mgammabold\}$ a basis of~$\Pbb_{\p-1}(\E)\setminus \Rbb$
such that each~$\mgammabold$ has zero average over~$\E$,
\begin{alignat*}{2}
&\begin{cases}
-\Deltabold \psiboldB_i - \nabla s =\mathbf 0                           & \text{in } \E\\
\div\psiboldB_i = \vert\E\vert^{-1} \int_{\partial \E} \phiboldB_i \cdot \nbfE    & \text{in } \E\\
\psiboldB_i = \phiboldB_i                                           & \text{on } \partial\E
\end{cases}
&& \qquad \forall i=1,\dots,2\p\cdot(\# \text{edges of~$\E$}),\\
&\begin{cases}
-\Deltabold \psiboldperp_{\boldalpha} - \nabla s =\xperp\malphabold & \text{in } \E\\
\div\psiboldperp_{\boldalpha} = 0                               & \text{in } \E\\
\psiboldperp_{\boldalpha} = \mathbf0                            & \text{on } \partial\E
\end{cases}
&& \qquad \forall \vert \boldalpha \vert =0,\dots,\p-3,  \\
&\begin{cases}
-\Deltabold \psibolddiv_{\boldgamma} - \nabla s = \mathbf 0 & \text{in } \E\\
\div\psibolddiv_{\boldgamma} = \mgammabold              & \text{in } \E\\
\psibolddiv_{\boldgamma} = \mathbf0                     & \text{on } \partial\E
\end{cases}
&& \qquad \forall \vert\boldgamma\vert=1,\dots,\p-1.
\end{alignat*}
The span of the three above set of functions is the space~$\VbfnE$.
To see this, it suffices to observe that the number of functions is equal to the dimension of~$\VbfnE$
and that they are independent of each other.
Therefore, we can write each basis function~$\phibold$ as a linear combination of the~$\psibold$ functions.
Furthermore, the~$\psibold$ functions can be approximated at any precision by employing finite elements on sufficiently fine triangulations of~$\E$.
In other words, if we have at hand any finite element approximation of the~$\psibold$ basis functions,
then we only have to write the~$\phibold$ functions in terms of the~$\psibold$ functions,
and then compute the resulting matrix~$\Bbf$.
This is what we detail in Sections~\ref{subsection:computation-boundary},
\ref{subsection:computation-orthogonal},
and~\ref{subsection:computation-div}.
The algorithm reads as follows:
\begin{enumerate}
\item approximate the $\psibold$ basis functions using a FEM triangulation on~$\E$;
\item find the $\phibold$ basis functions as a linear combination of the~$\psibold$ basis functions,
described in Sections~\ref{subsection:computation-boundary}, \ref{subsection:computation-orthogonal},
and~\ref{subsection:computation-div} below;
\item compute the matrices~$\Abf$ and~$\Bbf$;
\item solve the generalized eigenvalue problem~\eqref{generalized:eigenvalue-problem}.
\end{enumerate}
In Section~\ref{subsection:nr-p-version},
we check the behaviour of the minimum and maximum generalized eigenvalues of~\eqref{generalized:eigenvalue-problem}
on a fixed element and increasing the degree of accuracy~$\p$ of the scheme;
in Section~\ref{subsection:nr-badly-shaped},
we keep fixed the degree of accuracy, focus on two different types of elements,
deform them in different ways, and check the behaviour of the corresponding stability constants.

\subsection{Expanding the boundary-type functions} \label{subsection:computation-boundary}
For all~$i=1,\dots,2\p\cdot(\# \text{edges of~$\E$})$,
we write
\[
\phiboldB_i
= \sum_{j} \Ai_j \psiboldB_j
  + \sum_{\boldbeta} \Bi_{\boldbeta} \psiboldperp_{\boldbeta}
  + \sum_{\bolddelta} \Ci_{\bolddelta} \psibolddiv_{\bolddelta}.
\]
We have to determine the A, B, and C-type coefficients by imposing the DoFs definition.

First, we use the boundary DoFs.
For any vertex or Gau\ss-Lobatto node~$N_\E$ on any edge~$\e$ of~$\E$, we have
\[
\deltacal_{i,k}=\phiboldB_i (N_k)
= \sum_{j} \Ai_j \psiboldB_j (N_k) = \Ai_k
\qquad \forall k=1,\dots, 2\p\cdot(\# \text{edges of~$\E$}).
\]
Therefore, we get the simplified expression
\[
\phiboldB_i
= \psiboldB_i
 + \sum_{\boldbeta} \Bi_{\boldbeta} \psiboldperp_{\boldbeta}
 + \sum_{\bolddelta} \Ci_{\bolddelta} \psibolddiv_{\bolddelta}.
\]
To find the B and C-type coefficients, we impose the divergence and ``orthogonal'' DoFs definition.
First, we impose the ``orthogonal'' moments definition and write
\[
    0
    = \int_\E \phiboldB_i \cdot (\xperp \mbetaboldtilde)
    = \int_\E \psiboldB_i \cdot (\xperp \mbetaboldtilde)
      + \sum_{\boldbeta} \Bi_{\boldbeta} \int_\E \psiboldperp_{\boldbeta} \cdot (\xperp \mbetaboldtilde)
      + \sum_{\bolddelta} \Ci_{\bolddelta} \int_\E \psibolddiv_{\bolddelta} \cdot (\xperp \mbetaboldtilde),
\]
whence we deduce the conditions, for all~$\vert \boldbetatilde \vert =0,\dots,\p-3$,
\[
\sum_{\boldbeta} \left( \int_\E \psiboldperp_{\boldbeta} \cdot (\xperp \mbetaboldtilde) \right) \Bi_{\boldbeta} 
+ \sum_{\bolddelta} \left(\int_\E \psibolddiv_{\bolddelta} \cdot (\xperp \mbetaboldtilde) \right) \Ci_{\bolddelta} = -  \int_\E \psiboldB_i \cdot (\xperp \mbetaboldtilde) .
\]
Next, we impose the divergence moments definition,
recall that the test polynomial~$\mdeltaboldtilde$ has zero average over~$\E$, and get
\[
\begin{split}
0
& = \int_\E \div \phiboldB_i \mdeltaboldtilde
= \int_\E \div \psiboldB_i \mdeltaboldtilde +  \sum_{\bolddelta} \Ci_{\bolddelta} \int_\E \div \psibolddiv_{\bolddelta} \mdeltaboldtilde \\
& = \vert\E\vert^{-1} \int_{\partial \E} \psiboldB_i \cdot \nbfE \int_\E \mdeltaboldtilde
    +\sum_{\bolddelta} \Ci_{\bolddelta} \int_\E \mdeltabold \ \mdeltaboldtilde
= \sum_{\bolddelta} \Ci_{\bolddelta} \int_\E \mdeltabold \ \mdeltaboldtilde,
\end{split}
\]
whence we deduce the conditions, for all~$\vert \bolddeltatilde \vert =1,\dots,\p-1$,
\[
\sum_{\bolddelta} \left( \int_\E \mdeltabold \ \mdeltaboldtilde \right) \Ci_{\bolddelta} = 0.
\]
Using the coercivity of any polynomial mass matrix,
we deduce~$\Ci_{\bolddelta}=0$ for all~$i$ and~$\bolddelta$.

Thus, the B-type coefficients are obtained by solving the linear system
\[
\sum_{\boldbeta} \left( \int_\E \psiboldperp_{\boldbeta} \cdot (\xperp \mbetaboldtilde) \right) \Bi_{\boldbeta} 
= -  \int_\E \psiboldB_i \cdot (\xperp \mbetaboldtilde) .
\]

\subsection{Expanding the orthogonal-type functions} \label{subsection:computation-orthogonal}
For all~$\vert \boldalpha \vert=0,\dots,\p-3$, we write
\[
\phiboldperp_{\boldalpha}
=\sum_{j} \Aalpha_j \psiboldB_j
 + \sum_{\boldbeta} \Balpha_{\boldbeta} \psiboldperp_{\boldbeta}
 + \sum_{\bolddelta} \Calpha_{\bolddelta} \psibolddiv_{\bolddelta}.
\]
We have to determine the A, B, and C-type coefficients by imposing the DoFs definition.
Imposing the boundary DoFs definition, we readily obtain that~$\Aalpha_j=0$
for all~$j=1,\dots, 2\p \cdot(\# \text{edges of~$\E$})$.

Thus, we focus on the other coefficients.
First, we impose the ``orthogonal'' moments definition and write
\[
    \deltacal_{\boldalpha,\boldbetatilde} \vert\E\vert
    = \int_\E \phiboldperp_{\boldalpha} \cdot (\xperp \mbetaboldtilde)
    = \sum_{\boldbeta} \Balpha_{\boldbeta} \int_\E \psiboldperp_{\boldbeta} \cdot (\xperp \mbetaboldtilde)
     + \sum_{\bolddelta} \Calpha_{\bolddelta} \int_\E \psibolddiv_{\bolddelta} \cdot (\xperp \mbetaboldtilde),
\]
whence we deduce the conditions, for all~$\vert \boldbetatilde \vert =0,\dots,\p-3$,
\[
\sum_{\boldbeta} \left( \int_\E \psiboldperp_{\boldbeta} \cdot (\xperp \mbetaboldtilde) \right) \Balpha_{\boldbeta} 
+ \sum_{\bolddelta} \left(\int_\E \psibolddiv_{\bolddelta} \cdot (\xperp \mbetaboldtilde) \right) \Calpha_{\bolddelta} = \deltacal_{\boldalpha,\boldbetatilde} \vert\E\vert.
\]
Next, we impose the divergence moments definition and get
\[
0
=\int_\E \div \phiboldperp_{\boldalpha} \mdeltaboldtilde
= \sum_{\bolddelta} \Calpha_{\bolddelta} \int_\E \div \psibolddiv_{\bolddelta} \mdeltaboldtilde
= \sum_{\bolddelta} \Calpha_{\bolddelta} \int_\E \mdeltabold \ \mdeltaboldtilde,
\]
whence we deduce the conditions, for all~$\vert \bolddeltatilde \vert =1,\dots,\p-1$,
\[
\sum_{\bolddelta} \left( \int_\E \mdeltabold \ \mdeltaboldtilde \right) \Calpha_{\bolddelta} = 0.
\]
Using the coercivity of any polynomial mass matrix,
we deduce~$\Calpha_{\bolddelta}$ for all~$\boldalpha$ and~$\bolddelta$.

Thus, the B-type coefficients are obtained by solving the linear system
\[
\sum_{\boldbeta} \left( \int_\E \psiboldperp_{\boldbeta} \cdot (\xperp \mbetaboldtilde) \right) \Balpha_{\boldbeta} 
= \deltacal_{\boldalpha,\boldbetatilde} \vert\E\vert.
\]

\subsection{Expanding the divergence-type functions} \label{subsection:computation-div}
For all~$\vert \boldgamma \vert=1,\dots,\p-1$, we write
\[
\phibolddiv_{\boldgamma}
=\sum_{j} \Agamma_j \psiboldB_j
 + \sum_{\boldbeta} \Bgamma_{\boldbeta} \psiboldperp_{\boldbeta}
 + \sum_{\bolddelta} \Cgamma_{\bolddelta} \psibolddiv_{\bolddelta}.
\]
We have to determine the A, B, and C-type coefficients by imposing the DoFs definition.
Imposing the boundary DoFs definition, we readily obtain that~$\Agamma_j=0$
for all~$j=1,\dots, 2\p \cdot(\# \text{edges of~$\E$})$.

Thus, we focus on the other coefficients.
First, we impose the ``orthogonal'' moments definition and write
\[
    0
    = \int_\E \phibolddiv_{\boldgamma} \cdot (\xperp \mbetaboldtilde)
    = \sum_{\boldbeta} \Bgamma_{\boldbeta} \int_\E \psiboldperp_{\boldbeta} \cdot (\xperp \mbetaboldtilde)
     + \sum_{\bolddelta} \Cgamma_{\bolddelta} \int_\E \psibolddiv_{\bolddelta} \cdot (\xperp \mbetaboldtilde),
\]
whence we deduce the conditions, for all~$\vert \boldbetatilde \vert =0,\dots,\p-3$,
\begin{equation} \label{coupling:divergence-orthogonal}
\sum_{\boldbeta} \left( \int_\E \psiboldperp_{\boldbeta} \cdot (\xperp \mbetaboldtilde) \right) \Bgamma_{\boldgamma} 
+ \sum_{\bolddelta} \left(\int_\E \psibolddiv_{\bolddelta} \cdot (\xperp \mbetaboldtilde) \right) \Cgamma_{\bolddelta} = 0.
\end{equation}
Next, we impose the divergence moments definition and get
\[
\frac{\vert\E\vert}{\hE} \deltacal_{\boldgamma, \bolddeltatilde}
=\int_\E \div \phibolddiv_{\boldalpha} \mdeltaboldtilde
= \sum_{\bolddelta} \Cgamma_{\bolddelta} \int_\E \div \psibolddiv_{\bolddelta} \mdeltaboldtilde
= \sum_{\bolddelta} \Cgamma_{\bolddelta} \int_\E \mdeltabold \ \mdeltaboldtilde,
\]
whence we deduce the conditions, for all~$\vert \bolddeltatilde \vert =1,\dots,\p-1$,
\begin{equation} \label{coupling:divergence-divergence}
\sum_{\bolddelta} \left( \int_\E \mdeltabold \ \mdeltaboldtilde \right) \Cgamma_{\bolddelta} 
= \frac{\vert\E\vert}{\hE} \deltacal_{\boldgamma, \bolddeltatilde}.
\end{equation}
The B and C-type coefficients are obtained by solving the linear system
resulting from~\eqref{coupling:divergence-orthogonal} and~\eqref{coupling:divergence-divergence}.

\subsection{Stability constants when increasing the degree of accuracy} \label{subsection:nr-p-version}
In light of the above approximation of the virtual element basis functions,
we provide here the minimum (nonzero) and maximum eigenvalues
of the generalized eigenvalue problem~\eqref{generalized:eigenvalue-problem}
on a given pentagon.
Such constants clearly correspond to~$\lambda_{min}=c_*$
and~$\lambda_{max}=c^*$ in the stability bound
\[
c_* \aE(\vbfn,\vbfn)
\le \ahE(\vbfn,\vbfn)
\le c^* \aE(\vbfn,\vbfn)
\qquad \forall \vbfn \in \VbfnE\setminus \Rbb.
\]
Notably, we analyze the eigenvalues while increasing the degree of accuracy of the scheme.

This is relevant to check as all the stability estimates are proved via inverse estimates,
which typically depend on the polynomial degree and the degree of accuracy of the scheme.
This is also interesting to check
the value of the ``hidden'' constants in the theoretical bounds.
We employ the theoretical~\eqref{explicit:stabilization-2D}
and the dofi-dofi~\eqref{dofi-dofi:stab:2D} stabilizations
and show the results in Table~\ref{table:eigenvalue:p-version}.

\begin{table}
\centering
\setlength{\abovecaptionskip}{-0.01cm}
\caption{Minimum and maximum eigenvalues of the generalized eigenvalue problem~\eqref{generalized:eigenvalue-problem}
employing the theoretical~\eqref{explicit:stabilization-2D}
and the dofi-dofi~\eqref{dofi-dofi:stab:2D} stabilizations~$\SE(\cdot,\cdot)$ and~$\SE_{D}(\cdot,\cdot)$.
We fix a pentagonal element and increase the degree of accuracy~$\p$.}
\label{table:eigenvalue:p-version}
{
\begin{tabular}{ccc|ccccc}
&\multicolumn{2}{c|}{$S^{\E}(\cdot,\cdot)$}&\multicolumn{2}{c}{$S_{D}^{\E}(\cdot,\cdot)$}\\
\hline
 \multicolumn{1}{c}{$p$} 
 & \multicolumn{1}{c}{$\lambda_{\min}$}
 & \multicolumn{1}{c|}{$\lambda_{\max}$}
 & \multicolumn{1}{c}{$\lambda_{\min}$}
 & \multicolumn{1}{c}{$\lambda_{\max}$}\\
\hline
2 & 1.7265e-01 & 1.0037e+00 & 1.7098e-01 & 1.0037e+00 \\
3 & 1.5570e-01 & 2.4381e+01 & 1.3578e-01 & 2.4269e+01 \\
4 & 1.4606e-01 & 2.5338e+01 & 1.2649e-01 & 2.5206e+01 \\
5 & 1.5364e-01 & 8.0281e+01 & 1.3667e-01 & 8.0154e+01 \\
6 & 1.4909e-01 & 9.8691e+01 & 1.2966e-01 & 9.8519e+01 \\
7 & 1.3999e-01 & 2.2056e+02 & 1.1612e-01 & 2.1990e+02 \\
8 & 1.2007e-01 & 2.9917e+02 & 1.0388e-01 & 2.9877e+02 \\
9 & 1.1133e-01 & 4.8713e+02 & 8.7618e-02 & 4.8637e+02 \\
\end{tabular}}
\centering
\end{table}

From Table~\ref{table:eigenvalue:p-version},
we observe that the stability constants
depend on the degree of accuracy~$\p$ only in a moderate way.
This is not surprising as a similar behaviour was observed for the Poisson-type virtual element method
in~\cite[Table~1]{BeiraoDaVeiga-Chernov-Mascotto-Russo:2016}
and~\cite[Table~1]{BeiraoDaVeiga-Chernov-Mascotto-Russo:2018}.

We are further interested in checking the condition number of the matrices~$\Abf$ and~$\Bbf$ appearing in~\eqref{generalized:eigenvalue-problem};
the details are detailed in Table~\ref{table:condition-p-version} below.

\begin{table}
\centering
\setlength{\abovecaptionskip}{-0.05cm}
\caption{Condition numbers of the virtual element matrices~$\Abf$ and~$\Abf_{D}$
(computed with respect to the theoretical~\eqref{explicit:stabilization-2D} and dofi-dofi~\eqref{dofi-dofi:stab:2D} stabilizations)
and the (finite element approximation of the) stiffness matrix~$\Bbf$
for increasing degree of accuracy~$\p$.}
{
\begin{tabular}{@{\extracolsep{\fill}}cccccccccccccc}
\multicolumn{1}{c}{$p$}
&\multicolumn{1}{c}{2}
&\multicolumn{1}{c}{3}
&\multicolumn{1}{c}{4}
&\multicolumn{1}{c}{5}
&\multicolumn{1}{c}{6}
&\multicolumn{1}{c}{7}
&\multicolumn{1}{c}{8} 
&\multicolumn{1}{c}{9}\\
\hline
$\Abf$      & 8.76e+01 & 4.92e+03 & 4.84e+05 & 3.72e+07 & 2.35e+09 & 1.41e+11 & 8.32e+12 & 4.93e+14\\
\hline
$\Abf_{D}$  & 8.78e+01 & 4.91e+03 & 4.80e+05 & 3.66e+07 & 2.30e+09 & 1.37e+11 & 8.05e+12 & 4.75e+14 \\
\hline
$\Bbf$      & 9.43e+01 & 5.40e+03 & 5.75e+05 & 4.42e+07 & 2.79e+09 & 1.67e+11 & 9.80e+12 & 5.68e+14\\
\end{tabular}}
\label{table:condition-p-version}
\centering
\end{table}

From Table~\ref{table:condition-p-version},
we see that the condition number of the (finite element approximation of the) stiffness matrix~$\Bbf$
is always larger than that of the corresponding virtual element matrices.
In a sense, this suggests that approximating the exact stiffness matrix
by a consistency and a stabilization term as in~\eqref{local-dscrete-bf}
is not leading to larger condition numbers
but rather has a slight beneficial effect.

\subsection{Stability constants on sequences of badly shaped elements} \label{subsection:nr-badly-shaped}
In this section, we provide the minimum (nonzero) and maximum eigenvalues
of the generalized eigenvalue problem~\eqref{generalized:eigenvalue-problem}
for a fixed degree of accuracy.
Notably, we analyze the behaviour of the eigenvalues
on sequences of elements with degenerating geometry.

The first sequence we consider is constructed as in Figure~\ref{figure:lose-star-shapedness}:
the first element is a square with a hanging node;
the other elements are obtained by moving the hanging node towards the opposite edge.
It is easy to check that the star-shapedness constant of the sequence goes to zero.

\begin{figure}
\begin{center}
\begin{tikzpicture}[anchor = south]
\draw[black, very thick, -] (0,0)--(2,0)--(2,2)--(0,2)--(0,0);
\filldraw[black] (0,0) circle (1.5pt);
\filldraw[black] (2,0) circle (1.5pt);
\filldraw[black] (2,2) circle (1.5pt);
\filldraw[black] (0,2) circle (1.5pt);
\draw[black, ->] (1,2) -- (1,1);
\draw[black, -] (.3,.25) node[black] {\tiny{$\E_1$}};
\end{tikzpicture}
\qquad
\begin{tikzpicture}[anchor=south]
\draw[black, very thick, -] (0,0)--(2,0)--(2,2)--(1,1)--(0,2)--(0,0);
\filldraw[black] (0,0) circle (1.5pt);
\filldraw[black] (2,0) circle (1.5pt);
\filldraw[black] (2,2) circle (1.5pt);
\filldraw[black] (0,2) circle (1.5pt);
\filldraw[black] (1,1) circle (1.5pt);
\draw[black, ->] (1,1) -- (1,.5);
\draw[black, -] (.3,.25) node[black] {\tiny{$\E_2$}};
\end{tikzpicture}
\qquad
\begin{tikzpicture}[anchor=south]
\draw[black, very thick, -] (0,0)--(2,0)--(2,2)--(1,.5)--(0,2)--(0,0);
\filldraw[black] (0,0) circle (1.5pt);
\filldraw[black] (2,0) circle (1.5pt);
\filldraw[black] (2,2) circle (1.5pt);
\filldraw[black] (0,2) circle (1.5pt);
\filldraw[black] (1,.5) circle (1.5pt);
\draw[black, ->] (1,.5) -- (1,.25);
\draw[black, -] (.3,.25) node[black] {\tiny{$\E_3$}};
\end{tikzpicture}
\qquad
\begin{tikzpicture}[anchor=south]
\draw[black, very thick, -] (0,0)--(2,0)--(2,2)--(1,.25)--(0,2)--(0,0);
\filldraw[black] (0,0) circle (1.5pt);
\filldraw[black] (2,0) circle (1.5pt);
\filldraw[black] (2,2) circle (1.5pt);
\filldraw[black] (0,2) circle (1.5pt);
\filldraw[black] (1,.25) circle (1.5pt);
\draw[black, ->] (1,.25) -- (1,.125);
\draw[black, -] (.3,.25) node[black] {\tiny{$\E_4$}};
\end{tikzpicture}
\qquad
\begin{tikzpicture}[anchor=south]
\draw[black, very thick, -] (0,0)--(2,0)--(2,2)--(1,.125)--(0,2)--(0,0);
\filldraw[black] (0,0) circle (1.5pt);
\filldraw[black] (2,0) circle (1.5pt);
\filldraw[black] (2,2) circle (1.5pt);
\filldraw[black] (0,2) circle (1.5pt);
\filldraw[black] (1,.125) circle (1.5pt);
\draw[black, ->] (1,.125) -- (1,.0625);
\draw[black, -] (.3,.25) node[black] {\tiny{$\E_5$}};
\end{tikzpicture}
\end{center}
\caption{First sequence of badly shaped elements.
The first element is a square with a hanging node;
the other elements are obtained by moving the hanging node towards the opposite edge.}
\label{figure:lose-star-shapedness}
\end{figure}
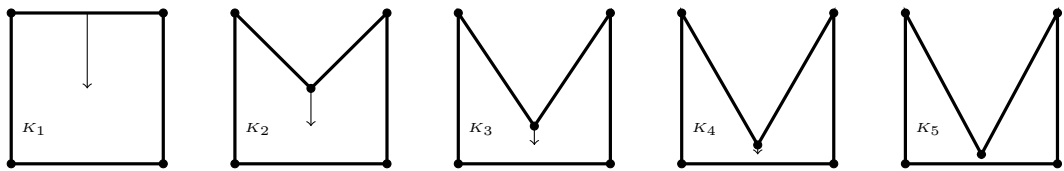

Next, we consider a sequence of elements obtained 
starting with a regular pentagon and halving the height
of the previous element in the sequence;
see Figure~\ref{figure:dilating-elements}.

\begin{figure}
\begin{center}
\begin{tikzpicture}[anchor = south]
\draw[black, very thick, -] (-.5,0)--(.5,0)--(1,1)--(0,2)--(-1,1)--(-.5,0);
\filldraw[black] (-.5,0) circle (1.5pt);
\filldraw[black] (.5,0) circle (1.5pt);
\filldraw[black] (1,1) circle (1.5pt);
\filldraw[black] (0,2) circle (1.5pt);
\filldraw[black] (-1,1) circle (1.5pt);
\draw[black, ->] (0,2) -- (0,1);
\draw[black, ->] (1,1) -- (1,.5);
\draw[black, ->] (-1,1) -- (-1,.5);
\draw[black, -] (-.5,1.6) node[black] {\tiny{$\E_1$}};
\end{tikzpicture}
\qquad
\begin{tikzpicture}[anchor=south]
\draw[black, very thick, -] (-.5,0)--(.5,0)--(1,.5)--(0,1)--(-1,.5)--(-.5,0);
\filldraw[black] (-.5,0) circle (1.5pt);
\filldraw[black] (.5,0) circle (1.5pt);
\filldraw[black] (1,.5) circle (1.5pt);
\filldraw[black] (0,1) circle (1.5pt);
\filldraw[black] (-1,.5) circle (1.5pt);
\draw[black, ->] (0,1) -- (0,.5);
\draw[black, ->] (1,.5) -- (1,.25);
\draw[black, ->] (-1,.5) -- (-1,.25);
\draw[black, -] (-.5,.8) node[black] {\tiny{$\E_2$}};
\end{tikzpicture}
\qquad
\begin{tikzpicture}[anchor=south]
\draw[black, very thick, -] (-.5,0)--(.5,0)--(1,.25)--(0,.5)--(-1,.25)--(-.5,0);
\filldraw[black] (-.5,0) circle (1.5pt);
\filldraw[black] (.5,0) circle (1.5pt);
\filldraw[black] (1,.25) circle (1.5pt);
\filldraw[black] (0,.5) circle (1.5pt);
\filldraw[black] (-1,.25) circle (1.5pt);
\draw[black, ->] (0,.5) -- (0,.25);
\draw[black, ->] (1,.25) -- (1,.125);
\draw[black, ->] (-1,.25) -- (-1,.125);
\draw[black, -] (-.5,.35) node[black] {\tiny{$\E_3$}};
\end{tikzpicture}
\qquad
\begin{tikzpicture}[anchor=south]
\draw[black, very thick, -] (-.5,0)--(.5,0)--(1,.125)--(0,.25)--(-1,.125)--(-.5,0);
\filldraw[black] (-.5,0) circle (1.5pt);
\filldraw[black] (.5,0) circle (1.5pt);
\filldraw[black] (1,.125) circle (1.5pt);
\filldraw[black] (0,.25) circle (1.5pt);
\filldraw[black] (-1,.125) circle (1.5pt);
\draw[black, ->] (0,.25) -- (0,.125);
\draw[black, ->] (1,.125) -- (1,.0625);
\draw[black, ->] (-1,.125) -- (-1,.0625);
\draw[black, -] (-.5,.2) node[black] {\tiny{$\E_4$}};
\end{tikzpicture}
\qquad
\begin{tikzpicture}[anchor=south]
\draw[black, very thick, -] (-.5,0)--(.5,0)--(1,.0625)--(0,.125)--(-1,.0625)--(-.5,0);
\filldraw[black] (-.5,0) circle (1.5pt);
\filldraw[black] (.5,0) circle (1.5pt);
\filldraw[black] (1,.0625) circle (1.5pt);
\filldraw[black] (0,.125) circle (1.5pt);
\filldraw[black] (-1,.0625) circle (1.5pt);
\draw[black, ->] (0,.125) -- (0,.0625);
\draw[black, ->] (1,.0625) -- (1,.03125);
\draw[black, ->] (-1,.0625) -- (-1,.03125);
\draw[black, -] (-.5,.1) node[black] {\tiny{$\E_5$}};
\end{tikzpicture}
\end{center}
\caption{First sequence of anisotropic elements.
The first element is a regular pentagon;
the other elements are obtained by halving the height of the previous polygon in the sequence.}
\label{figure:dilating-elements}
\end{figure}
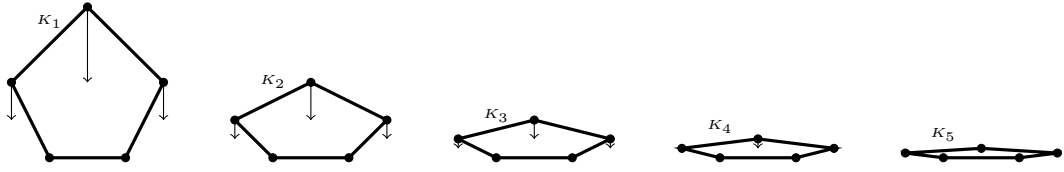

Both sequences do not satisfy the regularity assumptions in Section~\ref{section:introduction}.
For this reason, we cannot guarantee theoretically that the stability constants are robust with respect to the deformation of the elements.

In Table~\ref{table:badly-shaped}, we show the minimum (nonzero) and maximum generalized eigenvalues of~\eqref{generalized:eigenvalue-problem}
for~$\p=3$,
the theoretical~\eqref{explicit:stabilization-2D}
and the dofi-dofi~\eqref{dofi-dofi:stab:2D} stabilizations,
and the sequences of elements in Figures~\ref{figure:lose-star-shapedness} and~\ref{figure:dilating-elements}.

\begin{table}
\centering
\setlength{\abovecaptionskip}{-0.01cm}
\caption{Minimum and maximum eigenvalues of the generalized eigenvalue problem~\eqref{generalized:eigenvalue-problem}
employing the theoretical~\eqref{explicit:stabilization-2D}
and the dofi-dofi~\eqref{dofi-dofi:stab:2D} stabilizations~$\SE(\cdot,\cdot)$ and~$\SE_{D}(\cdot,\cdot)$.
We fix~$\p=3$ and consider the sequences of elements in Figures~\ref{figure:lose-star-shapedness} and~\ref{figure:dilating-elements}.}
\label{table:badly-shaped}
{
\begin{tabular}{@{\extracolsep{\fill}}ccc|ccccc}
&\multicolumn{2}{c|}{$S^{\E}(\cdot,\cdot)$}
&\multicolumn{2}{c}{$S_{D}^{\E}(\cdot,\cdot)$}\\
\hline
\multicolumn{1}{c}{$p=3$} 
& \multicolumn{1}{c} {$\lambda_{\min}$} 
& \multicolumn{1}{c|}{$\lambda_{\max}$}    
& \multicolumn{1}{c} {$\lambda_{\min}$}
& \multicolumn{1}{c} {$\lambda_{\max}$} \\
\hline
Figure~\ref{figure:lose-star-shapedness}& 1.7245e-01 & 2.4405e+01 & 1.5507e-01 & 2.4077e+01 \\
                                        & 2.4645e-02 & 2.6380e+01 & 2.2646e-02 & 2.5184e+01 \\
                                        & 2.0023e-02 & 5.2481e+01 & 1.9591e-02 & 5.0439e+01 \\
                                        & 1.1318e-02 & 8.3766e+01 & 1.1064e-02 & 8.0796e+01 \\
                                        & 6.2083e-03 & 1.1166e+02 & 6.0253e-03 & 1.0795e+02 \\
\hline
Figure~\ref{figure:dilating-elements}   & 1.1024e-01 & 2.9102e+01 & 1.0747e-01 & 2.9077e+01 \\
                                        & 3.5699e-02 & 5.7804e+01 & 3.5492e-02 & 5.7774e+01 \\
                                        & 8.6714e-03 & 1.7000e+02 & 8.6613e-03 & 1.6996e+02 \\
                                        & 1.9184e-03 & 6.1036e+02 & 1.9181e-03 & 6.1033e+02 \\
                                        & 4.9894e-04 & 2.3554e+03 & 4.9898e-04 & 2.3554e+03 \\
\end{tabular}}
\centering
\end{table}

From Table~\ref{table:badly-shaped}, we deduce that the stability constants are indeed deteriorating
together with the shape-regularity of the elements.

As in Section~\ref{subsection:nr-p-version},
we are interested in checking the condition number of the matrices~$\Abf$ and~$\Bbf$ appearing in~\eqref{generalized:eigenvalue-problem},
i.e., the condition number of the (finite element approximation of the) stiffness matrix~$\Bbf$ and its virtual element counterpart~$\Abf$;
the results are given in Table~\ref{table:condition-badly-shaped} below.

\begin{table}
\centering
\setlength{\abovecaptionskip}{-0.05cm}
\caption{Condition numbers of the
virtual element matrices~$\Abf$ and~$\Abf_{D}$
(computed with respect to the theoretical~\eqref{explicit:stabilization-2D} and dofi-dofi~\eqref{dofi-dofi:stab:2D} stabilizations)
and the (finite element approximation of the)} stiffness matrix~$\Bbf$
for degree of accuracy~$\p=3$ on the two sequences of elements
in Figures~\ref{figure:lose-star-shapedness} and~\ref{figure:dilating-elements}.
\label{table:condition-badly-shaped}
\begin{tabular}{@{\extracolsep{\fill}}cccccccccccccc}
$\Abf$ Fig.~\ref{figure:lose-star-shapedness}    & 2.9439e+03 & 1.3337e+04 & 1.8254e+04 & 2.1036e+04 & 2.2459e+04 \\
$\Abf$ Fig.~\ref{figure:dilating-elements}       & 2.2872e+04 & 3.1029e+05 & 4.6531e+06 & 7.2463e+07 & 1.1456e+09 \\
\hline
$\Abf_{D}$ Fig.~\ref{figure:lose-star-shapedness} & 2.9272e+03 & 1.3352e+04 & 1.8291e+04 & 2.1082e+04 & 2.2501e+04 \\
$\Abf_{D}$ Fig.~\ref{figure:dilating-elements}    & 2.2866e+04 & 3.1020e+05 & 4.6527e+06 & 7.2463e+07 & 1.1455e+09 \\
\hline
$\Bbf$ Fig.~\ref{figure:lose-star-shapedness}      & 3.6700e+03 & 3.2448e+04 & 1.9277e+05 & 3.2232e+05 & 4.1319e+05 \\
$\Bbf$ Fig.~\ref{figure:dilating-elements}         & 2.4815e+04 & 3.2299e+05 & 9.1850e+06 & 4.0054e+08 & 1.4112e+10 \\
\end{tabular}
\centering
\end{table}

From Table~\ref{table:condition-badly-shaped}, the condition number of the (finite element approximation of the) stiffness matrix~$\Bbf$
is always larger than that of the corresponding virtual element matrices.
Hence, similar comments as those for Table~\ref{table:condition-p-version} apply.

\section{Conclusions} \label{section:conclusions}
We investigated some open issues in the analysis of Stokes-like virtual element spaces.
Notably, we proved stability properties in two and three dimensions,
and furthermore derived interpolation estimates by simplifying the current state-of-the-art proofs.
Numerical experiments seem to indicate that the stability constants
only depend moderately on~$\p$;
however, they can degenerate more rapidly for nonregular element geometries.
On the other hand, such a discrepancy of the discrete form
with respect to the ``exact'' one
may be beneficial and explain why in this degenerate mesh conditions
the VEM often performs remarkably well.

\paragraph*{Funding.}
J. Meng has been supported by the China Scholarship Council (No. $202106280167$)
and the Fundamental Research Funds for the Central Universities (No. xzy $022019040$).
L. Beir\~ao da Veiga was partially supported by the Italian MIUR through the PRIN grants n. $905$ $201744$KLJL.
L. Mascotto acknowledges support from the Austrian Science Fund (FWF) Project P33477.

\paragraph*{Competing Interests.}
The authors have no relevant financial or non-financial interests to disclose.

\paragraph*{Data Availability.}
The datasets generated during and/or analysed during the current study are available on request.

{\footnotesize \bibliography{bibliogr}} \bibliographystyle{plain}

\begin{thebibliography}{10}

\bibitem{Adak-Mora-Natarajan-Silgado:2021}
D.~Adak, D.~Mora, S.~Natarajan, and A.~Silgado.
\newblock A virtual element discretization for the time dependent
  {N}avier--{S}tokes equations in stream-function formulation.
\newblock {\em ESAIM Math. Model. Numer. Anal.}, 55(5):2535--2566, 2021.

\bibitem{Aghili-Boyaval-DiPietro:2015}
J.~Aghili, S.~Boyaval, and D.~A. Di~Pietro.
\newblock Hybridization of mixed high-order methods on general meshes and
  application to the {S}tokes equations.
\newblock {\em Comput. Methods Appl. Math.}, 15(2):111--134, 2015.

\bibitem{Ahmad-Alsaedi-Brezzi-Marini-Russo:2013}
B.~Ahmad, A.~Alsaedi, F.~Brezzi, L.D. Marini, and A.~Russo.
\newblock {E}quivalent projectors for virtual element methods.
\newblock {\em Comput. Math. Appl.}, 66(3):376--391, 2013.

\bibitem{Antonietti-BeiraoDaVeiga-Mora-Verani:2014}
P.~F. Antonietti, L.~{Beir{\~a}o da Veiga}, D.~Mora, and M.~Verani.
\newblock A stream virtual element formulation of the {S}tokes problem on
  polygonal meshes.
\newblock {\em SIAM J. Numer. Anal.}, 52(1):386--404, 2014.

\bibitem{BeiraoDaVeiga-Brezzi-Cangiani-Manzini-Marini-Russo:2013}
L.~{Beir{\~a}o da Veiga}, F.~Brezzi, A.~Cangiani, G.~Manzini, L.D. Marini, and
  A.~Russo.
\newblock Basic principles of virtual element methods.
\newblock {\em Math. Models Methods Appl. Sci.}, 23(01):199--214, 2013.

\bibitem{BeiraoDaVeiga-Chernov-Mascotto-Russo:2016}
L.~{Beir{\~a}o da Veiga}, A.~Chernov, L.~Mascotto, and A.~Russo.
\newblock Basic principles of $hp$ virtual elements on quasiuniform meshes.
\newblock {\em Math. Models Methods Appl. Sci.}, 26(8):1567--1598, 2016.

\bibitem{BeiraoDaVeiga-Chernov-Mascotto-Russo:2018}
L.~{Beir{\~a}o da Veiga}, A.~Chernov, L.~Mascotto, and A.~Russo.
\newblock Exponential convergence of the $hp$ virtual element method with
  corner singularity.
\newblock {\em Numer. Math.}, 138(3):581--613, 2018.

\bibitem{BeiraoDaVeiga-Dassi-Manzini-Mascotto:2022}
L.~Beir\~ao~da Veiga, F.~Dassi, G.~Manzini, and L.~Mascotto.
\newblock The virtual element method for the {3D} resistive magnetohydrodynamic
  model.
\newblock \url{https://arxiv.org/abs/2201.04417}, 2022.

\bibitem{BeiraoDaVeiga-Dassi-Vacca:2020}
L.~{Beir{\~a}o da Veiga}, F.~Dassi, and G.~Vacca.
\newblock The {S}tokes complex for virtual elements in three dimensions.
\newblock {\em Math. Models Meth. Appl. Sci.}, 30(03):477--512, 2020.

\bibitem{BeiraoDaVeiga-Lovadina-Russo:2017}
L.~{Beir{\~a}o da Veiga}, C.~Lovadina, and A.~Russo.
\newblock Stability analysis for the virtual element method.
\newblock {\em Math. Models Methods Appl. Sci.}, 27(13):2557--2594, 2017.

\bibitem{BeiraoDaVeiga-Lovadina-Vacca:2017}
L.~{Beir{\~a}o da Veiga}, C.~Lovadina, and G.~Vacca.
\newblock Divergence free virtual elements for the {S}tokes problem on
  polygonal meshes.
\newblock {\em ESAIM Math. Model. Numer. Anal.}, 51(2):509--535, 2017.

\bibitem{BeiraoDaVeiga-Lovadina-Vacca:2018}
L.~{Beir{\~a}o da Veiga}, C.~Lovadina, and G.~Vacca.
\newblock Virtual elements for the {N}avier--{S}tokes problem on polygonal
  meshes.
\newblock {\em SIAM J. Numer. Anal.}, 56(3):1210--1242, 2018.

\bibitem{BeiraoDaVeiga-Mora-Vacca:2019}
L.~{Beir{\~a}o da Veiga}, D.~Mora, and G.~Vacca.
\newblock The {S}tokes complex for virtual elements with application to
  {N}avier--{S}tokes flows.
\newblock {\em J. Sci. Comput.}, 81(2):990--1018, 2019.

\bibitem{Bernardi-Maday:1992}
C.~Bernardi and Y.~Maday.
\newblock Polynomial interpolation results in {S}obolev spaces.
\newblock {\em J. Comput. Appl. Math.}, 43(1):53--80, 1992.

\bibitem{Boffi-Brezzi-Fortin:2013}
D.~Boffi, F.~Brezzi, and M.~Fortin.
\newblock {\em Mixed {F}inite {E}lement {M}ethods and {A}pplications},
  volume~44.
\newblock Springer Series in Computational Mathematics, 2013.

\bibitem{Botti-DiPietro-Droniou:2019}
L.~Botti, D.~A. Di~Pietro, and J.~Droniou.
\newblock A {H}ybrid {H}igh-{O}rder method for the incompressible
  {N}avier--{S}tokes equations based on {T}emam's device.
\newblock {\em J. Comput. Phys.}, 376:786--816, 2019.

\bibitem{Brenner-Sung:2018}
S.~C. Brenner and L.-Y. Sung.
\newblock Virtual element methods on meshes with small edges or faces.
\newblock {\em Math. Models Methods Appl. Sci.}, 268(07):1291--1336, 2018.

\bibitem{Burman-Delay-Ern:2021}
E.~Burman, G.~Delay, and A.~Ern.
\newblock An unfitted hybrid high-order method for the {S}tokes interface
  problem.
\newblock {\em IMA J. Numer. Anal.}, 41(4):2362--2387, 2021.

\bibitem{Caceres-Gatica:2017}
E.~C{\'a}ceres and G.~N. Gatica.
\newblock A mixed virtual element method for the pseudostress-velocity
  formulation of the {S}tokes problem.
\newblock {\em IMA J. Numer. Anal.}, 37(1):296--331, 2017.

\bibitem{Caceres-Gatica-Sequeira:2017}
E.~C{\'a}ceres, G.~N. Gatica, and F.~A. Sequeira.
\newblock A mixed virtual element method for the {B}rinkman problem.
\newblock {\em Math. Models Meth. Appl. Sci.}, 27(04):707--743, 2017.

\bibitem{Caceres-Gatica-Sequeira:2018}
E.~C{\'a}ceres, G.~N. Gatica, and F.~A. Sequeira.
\newblock A mixed virtual element method for quasi-{N}ewtonian {S}tokes flows.
\newblock {\em SIAM J. Numer. Anal.}, 56(1):317--343, 2018.

\bibitem{Cangiani-Gyrya-Manzini:2016}
A.~Cangiani, V.~Gyrya, and G.~Manzini.
\newblock The non-conforming virtual element method for the {S}tokes equations.
\newblock {\em SIAM J. Numer. Anal.}, 54(6):3411--3435, 2016.

\bibitem{Cao-Chen:2018}
S.~Cao and L.~Chen.
\newblock Anisotropic error estimates of the linear virtual element method on
  polygonal meshes.
\newblock {\em SIAM J. Numer. Anal.}, 56(5):2913--2939, 2018.

\bibitem{Chen-Huang:2018}
L.~Chen and J.~Huang.
\newblock Some error analysis on virtual element methods.
\newblock {\em Calcolo}, 55(1):1--23, 2018.

\bibitem{Chernov-Marcati-Mascotto:2021}
A.~Chernov, C.~Marcati, and L.~Mascotto.
\newblock $p$- and $hp$-virtual elements for the {S}tokes problem.
\newblock {\em Adv. Comp. Math.}, 47(24), 2021.

\bibitem{Cockburn-Kanschat-Schoetzau-Schwab:2002}
B.~Cockburn, G.~Kanschat, D.~Sch{\"o}tzau, and Ch. Schwab.
\newblock Local discontinuous {G}alerkin methods for the {S}tokes system.
\newblock {\em SIAM J. Numer. Anal.}, 40(1):319--343, 2002.

\bibitem{Cockburn-Sayas:2014}
B.~Cockburn and F.-J. Sayas.
\newblock Divergence-conforming {HDG} methods for {S}tokes flows.
\newblock {\em Math. Comp.}, 83(288):1571--1598, 2014.

\bibitem{Cockburn-Shi:2014}
B.~Cockburn and K.~Shi.
\newblock Devising {HDG} methods for {S}tokes flow: an overview.
\newblock {\em Comput. \& Fluids}, 98:221--229, 2014.

\bibitem{Dassi-Vacca:2020}
F.~Dassi and G.~Vacca.
\newblock Bricks for the mixed high-order virtual element method: {P}rojectors
  and differential operators.
\newblock {\em Appl. Numer. Math.}, 155:140--159, 2020.

\bibitem{DiPietro-Krell:2018}
D.~A. Di~Pietro and S.~Krell.
\newblock A hybrid high-order method for the steady incompressible
  {N}avier--{S}tokes problem.
\newblock {\em J. Sci. Comput.}, 74(3):1677--1705, 2018.

\bibitem{Frerichs-Merdon:2022}
D.~Frerichs and C.~Merdon.
\newblock Divergence-preserving reconstructions on polygons and a really
  pressure-robust virtual element method for the {S}tokes problem.
\newblock {\em IMA J. Numer. Anal.}, 42(1):597--619, 2022.

\bibitem{Gatica-Munar-Sequeira:2018}
G.~N. Gatica, M.~Munar, and F.~A. Sequeira.
\newblock A mixed virtual element method for the {N}avier-{S}tokes equations.
\newblock {\em Math. Models Methods Appl. Sci}, 28(14):2719--2762, 2018.

\bibitem{Irisarri-Hauke:2019}
D.~Irisarri and G.~Hauke.
\newblock Stabilized virtual element methods for the unsteady incompressible
  {N}avier--{S}tokes equations.
\newblock {\em Calcolo}, 56(4):38, 2019.

\bibitem{Liu-Chen:2019}
X.~Liu and Z.~Chen.
\newblock The nonconforming virtual element method for the {N}avier-{S}tokes
  equations.
\newblock {\em Adv. Comput. Math.}, 45(1):51--74, 2019.

\bibitem{Liu-Li-Chen:2017}
X.~Liu, J.~Li, and Z.~Chen.
\newblock A nonconforming virtual element method for the {S}tokes problem on
  general meshes.
\newblock {\em Comput. Methods Appl. Mech. Engrg.}, 320:694--711, 2017.

\bibitem{Monk:2003}
P.~Monk.
\newblock {\em Finite {E}lement {M}ethods for {M}axwell's {E}quations}.
\newblock Oxford University Press, 2003.

\bibitem{Mora-Reales-Silgado:2021}
D.~Mora, C.~Reales, and A.~Silgado.
\newblock A $\mathcal{C}^1$ virtual element method of high order for the
  {B}rinkman equations in stream function formulation with pressure recovery.
\newblock {\em IMA J. Numer. Anal.}, 2021.
\newblock \url{https://doi.org/10.1093/imanum/drab078}.

\bibitem{Mora-Silgado:2022}
D.~Mora and A.~Silgado.
\newblock A $\mathcal{C}^1$ virtual element method for the stationary
  quasi-geostrophic equations of the ocean.
\newblock {\em Comput. Math. Appl.}, 116:212--228, 2022.

\bibitem{Qiu-Shi:2016}
W.~Qiu and K.~Shi.
\newblock A superconvergent {HDG} method for the incompressible
  {N}avier--{S}tokes equations on general polyhedral meshes.
\newblock {\em IMA J. Numer. Anal.}, 36(4):1943--1967, 2016.

\bibitem{Vacca:2018}
G.~Vacca.
\newblock An ${H}^1$-conforming virtual element for {D}arcy and {B}rinkman
  equations.
\newblock {\em Math. Models Methods Appl. Sci.}, 28(01):159--194, 2018.

\bibitem{Verfuerth:1999}
R.~Verf\"{u}rth.
\newblock A note on polynomial approximation in {S}obolev spaces.
\newblock {\em Math. Model. Numer. Anal.}, 33(4):715--719, 1999.

\bibitem{Wang-Wang-Chen-He:2019}
G.~Wang, F.~Wang, L.~Chen, and Y.~He.
\newblock A divergence free weak virtual element method for the
  {S}tokes--{D}arcy problem on general meshes.
\newblock {\em Comput. Methods Appl. Mech. Engrg.}, 344:998--1020, 2019.

\end{thebibliography}

\end{document}